\documentclass[a4paper,12pt]{amsart}
\usepackage{amssymb}
\usepackage[english]{babel}
\usepackage{color}

\newtheorem{theorem}{Theorem}[section]
\newtheorem{corollary}[theorem]{Corollary}
\newtheorem{lemma}[theorem]{Lemma}
\newtheorem{proposition}[theorem]{Proposition}

\begin{document}

\title[Innerness of continuous derivations]{Innerness of continuous derivations on algebras of locally measurable operators}
\author{A. F. Ber}
\address{Department of Mathematics,National University of Uzbekistan,
Vuzgorodok, 100174, Tashkent, Uzbekistan}
\email{ber@ucd.uz}

\author{V. I. Chilin }
\address{Department of Mathematics, National University of Uzbekistan,
Vuzgorodok, 100174, Tashkent, Uzbekistan}
\email{chilin@ucd.uz}

\author{F. A. Sukochev}
\address{School of Mathematics and Statistics, University of New South Wales, Sydney, NSW 2052, Australia }
\email{f.sukochev@unsw.edu.au}
\date{\today}
\begin{abstract}
It is established that every derivation continuous with respect to the local measure topology acting on the $*$-algebra $LS(\mathcal{M})$ of all locally measurable operators affiliated with a von Neumann algebra $\mathcal{M}$  is necessary inner. If $\mathcal{M}$ is a properly infinite von Neumann algebra, then every derivation on $LS(\mathcal{M})$ is inner. In addition, it is proved that any derivation on $\mathcal{M}$ with values in Banach $\mathcal{M}$-bimodule of locally measurable operators is  inner.
\end{abstract}


\maketitle

\section{Introduction}

One of the important results of the theory of derivations in
Banach bimodules is the Theorem of J. R. Ringrose on automatic
continuity of every derivation from a $C^*$-algebra $\mathcal{M}$
into a Banach $\mathcal{M}$-bimodule \cite{Ringrose}. This theorem
extends the well-known result that every derivation of a
$C^*$-algebra $\mathcal{M}$ is automatically norm continuous
\cite{Sak}. In the case when $\mathcal{M}$ is a $AW^*$-algebra (in
particular, $W^*$-algebra), every derivation on $\mathcal{M}$ is
inner \cite{Olesen}, \cite{Sak}.
Numerous results on continuity of derivations in Banach algebras are given in
\cite{Dales}.

Significant examples of $W^*$-modules are non-commutative
rearrangement invariant spaces of measurable operators affiliated with a von Neumann algebra. At the present time the theory of rearrangement
invariant spaces is actively developed
\cite{DDdP}, \cite{K-S}, and it gives useful applications both in the geometry of Banach spaces and in the theory of unbounded operators. Every non-commutative rearrangement invariant space
is a  solid linear space in the $*$-algebra $S(\mathcal{M},\tau)$
of all $\tau$-measurable operators affiliated with a von Neumann algebra $\mathcal{M}$, where $\tau$ is a faithful normal semifinite trace on $\mathcal{M}$ \cite{Ne}. The algebra $S(\mathcal{M},\tau)$ equipped with the natural topology  $t_\tau$ of convergence in measure generated by the trace  $\tau$ is a complete metrizable topological algebra. In its turn the algebra $S(\mathcal{M},\tau)$ represents a solid $*$-subalgebra of the $*$-algebra $LS(\mathcal{M})$ of
all locally measurable operators, affiliated with a von Neumann
algebra $\mathcal{M}$ \cite{San}, \cite{Yead}. The $*$-algebras
$LS(\mathcal{M})$ and $S(\mathcal{M},\tau)$ as well as the algebra
$S(\mathcal{M})$ of measurable operators affiliated with
$\mathcal{M}$ \cite{Seg}, are useful examples of
$EW^*$-algebras of unbounded operators \cite{Dixon1}. Moreover, in \cite{CZ} it is established that every $EW^*$-algebra $\mathcal{A}$ with the bounded part $\mathcal{A}_b=\mathcal{M}$ is a solid
$*$-subalgebra in the $*$-algebra $LS(\mathcal{M})$, i.e.
$LS(\mathcal{M})$ is the greatest $EW^*$-algebra of all
$EW^*$-algebras with the bounded part coinciding with
$\mathcal{M}$. In addition, the algebra $LS(\mathcal{M})$ with the natural topology  $t(\mathcal{M})$ of convergence locally in measure is a complete topological $*$-algebra \cite{Yead}.

Every $EW^*$-algebra $(\mathcal{A},t(\mathcal{M}))$ is an example of $GB^*$-algebras, which properties are well investigated in \cite{Dixon2}. The bounded part $\mathcal{A}(\mathcal{B}_0)$ of every $GB^*$-algebra $\mathcal{A}$ is a $C^*$-algebra
\cite{Dixon2} and the algebra $\mathcal{A}$ by itself is a topological bimodule over $\mathcal{A}(\mathcal{B}_0)$. In the case when
$\mathcal{A}$ is an $EW^*$-algebra the bounded part
$\mathcal{A}(\mathcal{B}_0)$ coincides with $\mathcal{A}_b$.
A natural development of J. R. Ringrose \cite{Ringrose} and
Sakai-Kadison Theorems \cite{Sak} is the study of the properties of
continuity and innerness of derivations acting from
$\mathcal{A}(\mathcal{B}_0)$ into $\mathcal{A}$.

This problem is directly connected with researches on derivations
on algebras of unbounded operators. One of the first work in this
field became the paper by C. Brodel and G. Lassher \cite{BL},
where it was established that every derivation on a complete
$O^*$-algebra of unbounded operator is spatial. Similar results
for other classes of locally convex algebras of unbounded
operators are obtained in \cite{Inoue, Inoue2}

Since every $EW^*$-algebra $\mathcal{A}$ with the bounded part
$\mathcal{A}_b=\mathcal{M}$ is a topological $*$-algebra of unbounded operators with respect to the non locally convex topology
$t(\mathcal{M})$, the problem of innerness and
$t(\mathcal{M})$-continuity of a derivation from
$\mathcal{M}$ into $\mathcal{A}$ seems natural.

In \cite{BCS3} it is proven that each derivation
$\delta:\mathcal{M}\rightarrow \mathcal{A}$ extends up to a
derivation from  $LS(\mathcal{M})$ into $LS(\mathcal{M})$. In this
respect, we should describe properties of
$t(\mathcal{M})$-continuity and innerness of derivations
$\delta:LS(\mathcal{M})\rightarrow LS(\mathcal{M})$.

In the setting of commutative $W^*$-algebras this problem  is
fully resolved in \cite{BCS2}. In the setting of von Neumann
algebras of type $I$, a thorough treatment of this problem may be
found in \cite{AAK} and \cite{BdPS}. The papers \cite{AAK,BCS2}
contain examples of non-inner derivations of the $*$-algebra
$LS(\mathcal{M})$, which are not continuous with respect to the
topology $t(\mathcal{M})$ of convergence locally in measure on
$LS(\mathcal{M})$. On the other hand, it is shown in \cite{AAK}
that in the special case when $\mathcal{M}$ is a properly infinite
von Neumann algebra of type $I$, every derivation of
$LS(\mathcal{M})$ is continuous with respect to the local measure
topology  $t(\mathcal{M})$. Using a completely different
technique, a similar result was also obtained in \cite{BdPS} under
the additional assumption that the predual space $\mathcal{M}_*$
to $\mathcal{M}$ is separable. It is of interest to observe that
an analogue of this result (that is the continuity of an arbitrary
derivation of $(LS(\mathcal{M}),t(\mathcal{M}))$) also holds for
any von Neumann algebra $\mathcal{M}$ of type $III$ \cite{AK}. In
\cite{AK} the following problem is formulated (Problem 3): Let
$\mathcal{M}$ be a  von Neumann algebra  of type $II$ and let
$\tau$ be a faithful normal semifinite trace on $\mathcal{M}$. Is
any derivation on a $*$-algebra $S(\mathcal{M},\tau)$ equipped
with the measure topology $t_\tau$ necessarily continuous? In
\cite{Ber} this problem is solved affirmatively for a properly
infinite algebra  $\mathcal{M}$. In view of the example we
mentioned above,   a natural  problem (similar to Problem 3 from
\cite{AK}) is whether any derivation of a $*$-algebra
$LS(\mathcal{M})$ is necessarily continuous with respect to the
topology $t(\mathcal{M})$, where $\mathcal{M}$ is a properly
infinite von Neumann algebra of type $II$. In \cite{BCS3} it is
given the positive solution of this problem. In fact, in
\cite{BCS3} it is established a much stronger result that any
derivation $\delta: \mathcal{A}\rightarrow LS(\mathcal{M})$,
where $\mathcal{A}$ is any $EW^*$-subalgebra in $LS(\mathcal{M})$
with $\mathcal{A}_b=\mathcal{M}$, is necessarily continuous with
respect to the topology $t(\mathcal{M})$ in the case when $\mathcal{M}$ is properly infinite von Neumann algebra.

In this respect the problem of innerness of
$t(\mathcal{M})$-continuous derivations
$\delta:\mathcal{A}\rightarrow\mathcal{A}$ ($EW^*$-version of
Sakai-Kadison Theorem) naturally arises. For a von Neumann algebra
of type $I$ and $III$ this problem is solved in \cite{AAK,AK}. In
the present paper it is proven that for every
$t(\mathcal{M})$-continuous derivation $\delta$ acting on an
$EW^*$-algebra $\mathcal{A}$ with the bounded part
$\mathcal{A}_b=\mathcal{M}$, there exists  $a\in\mathcal{A}$, such
that $\delta(x)=ax-xa=[a,x]$ for all $x\in\mathcal{A}$, i.e.
derivation $\delta$ is inner. Moreover, it is established that every derivation on a von Neumann algebra $\mathcal{M}$ with values in a noncommutative rearrangement invariant space
$\mathcal{E}\subset S(\mathcal{M},\tau)$ is necessary inner.

The proof proceeds in several stages. In section 3 we introduce
the notion of $\lambda$-system for a self-adjoint derivation
$\delta: LS(\mathcal{M})\rightarrow LS(\mathcal{M})$ and study the
properties of this $\lambda$-system, in particular, it is given
the estimate  on the value of dimensional function for the support
of $\lambda$-system. After that in section 4 it is given the proof
of the main result of the present paper (Theorem \ref{t1}) on
innerness of every $t(\mathcal{M})$-continuous derivation $\delta:
LS(\mathcal{M})\rightarrow LS(\mathcal{M})$. In particular, in
view of the result of \cite{BCS3}, it is shown that for a properly
infinite von Neumann algebra $\mathcal{M}$ every derivation on
$LS(\mathcal{M})$ is inner.  In section 5 we give applications of
Theorem \ref{t1}, establishing the innerness of all
$t(\mathcal{M})$-continuous derivations acting on an
$EW^*$-algebra $\mathcal{A}$ with the bounded part
$\mathcal{A}_b=\mathcal{M}$. In last section 6 we introduce the
class of Banach $\mathcal{M}$-bimodules of locally measurable
operators $\mathcal{E}\subset LS(\mathcal{M})$. This class
contains all noncommutative rearrangement invariant spaces. It is
proven that every derivation
$\delta:\mathcal{M\rightarrow\mathcal{E}}$ is inner, i.e. it has a
form  $\delta(x)=[d,x]=\delta_d(x)$ for all $x\in\mathcal{M}$ and
some $d\in\mathcal{E}$, in particular, $\delta$ is a continuous
derivation from $(\mathcal{M},\|\cdot\|_{\mathcal{M}})$ into
$(\mathcal{E},\|\cdot\|_{\mathcal{E}})$. In addition, the operator
$d\in\mathcal{E}$ such that $\delta=\delta_d$ may be chosen so
that  $\|d\|_{\mathcal{E}}\leq
2\|\delta\|_{\mathcal{M}\rightarrow\mathcal{E}}$.

We use terminology and notations from the von Neumann algebra
theory \cite{KR, Sak, Tak} and the theory of locally measurable
operators from \cite{MCh, San, Yead}.

\section{Preliminaries}

Let $H$ be a Hilbert space, let $B(H)$ be the $*$-algebra of all
bounded linear operators on $H$, and let $\mathbf{1}$ be the
identity operator on $H$. Given a von Neumann algebra
$\mathcal{M}$ acting on $H$, denote by $\mathcal{Z}(\mathcal{M})$
the centre of $\mathcal{M}$ and by
$\mathcal{P}(\mathcal{M})=\{p\in\mathcal{M}:\ p=p^2=p^*\}$ the
lattice of all projections in $\mathcal{M}$. Let
$\mathcal{P}_{fin}(\mathcal{M})$ be the set of all finite projections in
$\mathcal{M}$.

A linear operator $x:\mathfrak{D}\left( x\right) \rightarrow H $,
where the domain $\mathfrak{D}\left( x\right) $ of $x$ is a linear
subspace of $H$, is said to be {\it affiliated} with $\mathcal{M}$
if $yx\subseteq xy$ for all $y$ from the commutant
$\mathcal{M}^{\prime }$ of algebra $\mathcal{M}$.

A densely-defined closed linear operator $x$ (possibly unbounded)
affiliated with $\mathcal{M}$  is said to be \emph{measurable}
with respect to $\mathcal{M}$ if there exists a sequence
$\{p_n\}_{n=1}^\infty\subset \mathcal{P}(\mathcal{M})$ such that
$p_n\uparrow \mathbf{1},\ p_n(H)\subset \mathfrak{D}(x)$ and
$p_n^\bot=\mathbf{1}-p_n\in \mathcal{P}_{fin}(\mathcal{M})$ for every
$n\in\mathbb{N}$, where
$\mathbb{N}$ is the set of all natural numbers. Let us denote by
$S(\mathcal{M})$ the set of all measurable operators.

Let $x,y\in S(\mathcal{M})$. It is well known that $x+y$ and $xy$
are densely-defined and preclosed operators. Moreover, the
closures $\overline{x+y}$ (strong sum), $\overline{xy}$ (strong
product) and $x^*$ are also measurable, and equipped with this
operations (see \cite{Seg}) $S(\mathcal{M})$ is a unital
$*$-algebra over the field $\mathbb{C}$ of complex numbers. It is
clear that $\mathcal{M}$ is a $*$-subalgebra of $S(\mathcal{M})$.

A densely-defined linear operator $x$ affiliated with
$\mathcal{M}$ is called \emph{locally measurable} with respect to
$\mathcal{M}$ if there is a sequence $\{z_n\}_{n=1}^\infty$ of
central projections in $\mathcal{M}$ such that $z_n \uparrow
\mathbf{1},\ z_n(H)\subset\mathfrak{D}(x)$ and $xz_n\in
S(\mathcal{M})$ for all $n\in\mathbb{N}$.

The set $LS(\mathcal{M})$ of all locally measurable operators
(with respect to $\mathcal{M}$) is a unital $*$-algebra over the
field $\mathbb{C}$ with respect to the same algebraic operations
as in $S(\mathcal{M})$ \cite{Yead} and $S(\mathcal{M})$ is a
$*$-subalgebra of $LS(\mathcal{M})$. If $\mathcal{M}$ is finite,
or if $\dim(\mathcal{Z}(\mathcal{M}))<\infty$, the algebras
$S(\mathcal{M})$ and $LS(\mathcal{M})$ coincide \cite[ Corollary
2.3.5 and Theorem 2.3.16]{MCh}. If a von Neumann algebra
$\mathcal{M}$ is of type $III$ and
$\dim(\mathcal{Z}(\mathcal{M}))=\infty$, then
$S(\mathcal{M})=\mathcal{M}$ and $LS(\mathcal{M})\neq \mathcal{M}$
\cite[Theorem 2.2.19, Corollary 2.3.15]{MCh}.

For every $x\in S(\mathcal{Z}(\mathcal{M}))$ there exists a
sequence
$\{z_n\}_{n=1}^\infty\subset\mathcal{P}(\mathcal{Z}(\mathcal{M}))$
such that $z_n \uparrow \mathbf{1}$ and $xz_n\in\mathcal{M}$ for
all $n\in\mathbb{N}$. This means that $x\in LS(\mathcal{M})$.
Hence, $S(\mathcal{Z}(\mathcal{M}))$ is a $*$-subalgebra of
$LS(\mathcal{M})$ and $S(\mathcal{Z}(\mathcal{M}))$ coincides with
the center  of the $*$-algebra $LS(\mathcal{M})$.

For every subset $E\subset LS(\mathcal{M})$, the sets of all
self-adjoint (resp., positive) operators in $E$ will be denoted by
$E_h$ (resp. $E_+$). The partial order in $LS(\mathcal{M})$ is
defined by its cone $LS_+(\mathcal{M})$ and is denoted by $\leq$.

Let $\{z_i\}_{i\in I}$ be a family of pairwise
orthogonal non-zero central projections from $\mathcal{M}$ with
$\sup_{i\in I}z_i=\mathbf{1}$, where $I$ is an arbitrary set of
indixes (in this case, the family $\{z_i\}_{i\in I}$ is called a
central decomposition of the unity  $\mathbf{1}$). Consider the
$*$-algebra $\prod_{i\in I}LS(z_i\mathcal{M})$ with the
coordinate-wise operations and involution and for every $x\in LS(\mathcal{M})$ set
$$\phi(x):=\{z_ix\}_{i\in I}.$$
In \cite{Saito} it is proven that the mapping $\phi$ is a $*$-isomorphism from the $*$-algebra
$LS(\mathcal{M})$ onto $\prod_{i\in I}LS(z_i\mathcal{M})$. From here immediately follows

\begin{proposition}
\label{p1}
Given any central decomposition $\{z_i\}_{i\in I}$
of the unity and any family of elements $\{x_i\}_{i\in I}$ in
$LS(\mathcal{M})$ there exists a unique element $x\in
LS(\mathcal{M})$ such that $z_ix=z_ix_i$ for all $i\in I$.
\end{proposition}

Let $x$ be a closed operator with the dense domain $\mathfrak{D}(x)$
in $H$, let $x=u|x|$ be the polar decomposition of the operator
$x$, where $|x|=(x^*x)^{\frac{1}{2}}$ and $u$ is a  partial
isometry in $B(H)$ such that $u^*u$ is the right support $r(x)$ of
$x$. It is known that $x\in LS(\mathcal{M})$ (respectively, $x\in
S(\mathcal{M})$) if and only if $|x|\in LS(\mathcal{M})$
(respectively, $|x|\in S(\mathcal{M})$) and $u\in
\mathcal{M}$~\cite[\S\S\,2.2, 2.3]{MCh}. If $x$ is a self-adjoint
operator affiliated with $\mathcal{M}$, then the spectral family
of projections $\{E_\lambda(x)\}_{\lambda\in
  \mathbb{R}}$ for $x$ belongs to
$\mathcal{M}$~\cite[\S\,2.1]{MCh}. A locally measurable operator $x$
is measurable if and only if
$E_\lambda^\bot(|x|)\in \mathcal{P}_{fin}(\mathcal{M})$ for some $\lambda>0$ \cite[\S\,2.2]{MCh}.

Recall that two projections $e,f\in\mathcal{P}(\mathcal{M})$ are called equivalent (notation: $e\sim f$) if there exists a partial isometry  $u\in\mathcal{M}$ such that $u^*u=e$ and $uu^*=f$. For every operator $x\in LS(\mathcal{M})$ the left support  $l(x)$ and the right support $r(x)$ are always equivalent \cite[Prop.2.1.7(iii)]{MCh}, in addition  $r(|x|)=r(x)=l(x^*)$.
For projections $e,f\in\mathcal{P}(\mathcal{M})$ notation $e\preceq f$ means that there exists a projection $q\in\mathcal{P}(\mathcal{M})$ such that
$e\sim q\leq f$.

Now, let us recall the definition of the local measure topology.
Firstly, let $\mathcal{M}$ be a commutative von Neumann algebra. Then
$\mathcal{M}$ is $*$-isomorphic to the $*$-algebra
$L^\infty(\Omega,\Sigma,\mu)$ of all essentially bounded
measurable complex-valued functions defined on a measure space
$(\Omega,\Sigma,\mu)$ with the measure $\mu$ satisfying the direct
sum property (we identify functions that are equal almost
everywhere) (see e.g. \cite[Ch. III, \S 1]{Tak}).  The direct sum property of a measure $\mu$ means
that the Boolean algebra of all projections of the $*$-algebra
$L^\infty(\Omega,\Sigma,\mu)$ is order complete, and for any
non-zero $ p\in \mathcal{P}(\mathcal{M})$ there exists a non-zero
projection $q\leq p$ such that $\mu(q)<\infty$.  The direct sum property
of a measure $\mu$ is equivalent to the fact that the functional
$\tau(f):=\int_\Omega f\,d\mu$ is a semi-finite normal faithful trace on the algebra $L^\infty(\Omega,\Sigma,\mu)$.

Consider the $*$-algebra
$LS(\mathcal{M})=S(\mathcal{M})=L^0(\Omega,\Sigma,\mu)$ of all
measurable almost everywhere finite complex-valued functions
defined on $(\Omega,\Sigma,\mu)$ (functions that are equal almost
everywhere are identified).  Define on $L^0(\Omega,\Sigma,\mu)$
the local measure topology $t(L^\infty(\Omega))$, that is, the
Hausdorff vector topology, whose base of neighbourhoods of zero is
given by
$$
  W(B,\varepsilon,\delta):= \{f\in\ L^0(\Omega,\, \Sigma,\, \mu)
  \colon
  \ \hbox{there exists a set} \ E\in \Sigma\
  \mbox{such that}
  $$
  $$
   E\subseteq B, \ \mu(B\setminus
  E)\leq\delta, \ f\chi_E \in L^\infty(\Omega,\Sigma,\mu), \
  \|f\chi_E\|_{{L^\infty}(\Omega,\Sigma,\mu)}\leq\varepsilon\},
$$
where $\varepsilon, \ \delta >0$, $B\in\Sigma$, $\mu(B)<\infty$, $\chi(\omega)=1,\ \omega\in E$
and $\chi(\omega)=0$, when $\omega\notin E$.

Convergence of a net $\{f_\alpha\}$ to $f$ in the topology
$t(L^\infty(\Omega))$, denoted by $f_\alpha
\stackrel{t(L^\infty(\Omega))}{\longrightarrow}f$, means that
$f_\alpha \chi_B \rightarrow f\chi_B$ in measure $\mu$ for every
$B\in \Sigma$ with $\mu(B)<\infty$. Note, that the topology
$t(L^\infty(\Omega))$ does not change if the measure $\mu$ is
replaced with an equivalent measure \cite{Yead}.

Now let $\mathcal{M}$ be an arbitrary von Neumann algebra and let
$\varphi$ be a $*$-isomorphism from $\mathcal{Z}(\mathcal{M})$
onto the $*$-algebra $L^\infty(\Omega,\Sigma,\mu)$, where $\mu$ is
a measure satisfying the direct sum property.  Denote by
$L^+(\Omega,\, \Sigma,\, m)$ the set of all measurable real-valued
functions defined on $(\Omega,\Sigma,\mu)$ and taking values in
the extended half-line $[0,\, \infty]$ (functions that are equal
almost everywhere are identified). It was shown in~\cite{Seg} that
there exists a mapping
$$
\mathcal{D}\colon
\mathcal{P}(\mathcal{M})\to L^+(\Omega,\Sigma,\mu)
$$
that possesses the following properties:
\begin{itemize}
\item[(D1)]  $\mathcal{D}(p)\in L_+^0(\Omega,\Sigma,\mu)\Longleftrightarrow p\in \mathcal{P}_{fin}(\mathcal{M})$;
\item[(D2)] $\mathcal{D}(p\vee q)=\mathcal{D}(p)+\mathcal{D}(q)$ if
  $pq=0$;
\item[(D3)] $\mathcal{D}(u^*u)=\mathcal{D}(uu^*)$ for any partial
  isometry $u\in \mathcal{M}$;
\item[(D4)] $\mathcal{D}(zp)=\varphi(z)\mathcal{D}(p)$ for any $z\in
  \mathcal{P}(\mathcal{Z}(\mathcal{M}))$ and $p\in
  \mathcal{P}(\mathcal{M})$;
\item[(D5)] if $p_\alpha, p\in
  \mathcal{P}(\mathcal{M})$, $\alpha\in A$ and $p_\alpha\uparrow p$, then
  $\mathcal{D}(p)=\sup\limits_{\alpha\in A}\mathcal{D}(p_\alpha)$.
\end{itemize}

A mapping $\mathcal{D}\colon \mathcal{P}(\mathcal{M})\to
L^+(\Omega,\Sigma,\mu)$ that satisfies properties (D1)---(D5) is
called a \textit{dimension function} on $\mathcal{P}(\mathcal{M})$.

A dimension function $\mathcal{D}$ also has the following properties \cite{Seg}:
\begin{itemize}
\item[(D6)] if $p_n\in\mathcal{P}(\mathcal{M})$, $n\in\mathbb{N}$,
then $\mathcal{D}(\sup_{n\geq 1} p_n)\leq\sum_{n=1}^\infty\mathcal{D}(p_n)$, in addition,  when $p_np_m=0$, $n\neq m$, the equality holds;
\item[(D7)] if $p_n\in\mathcal{P}_{fin}(\mathcal{M})$, $n\in\mathbb{N}$, $p_n\downarrow 0$, then $\mathcal{D}(p_n)\rightarrow 0$ almost everywhere.
\end{itemize}

For arbitrary scalars $\varepsilon , \delta >0$ and a set $B\in
\Sigma$, $\mu(B)<\infty$, we set
\begin{equation}
\label{v1}
\begin{split}
  V(B,\varepsilon, \delta ) := \{x\in LS(\mathcal{M})\colon
  \mbox{there exist} \ p\in \mathcal{P}(\mathcal{M}),\\
  z\in \mathcal{P}(\mathcal{Z}(\mathcal{M})),\
   \mbox{such that}\ xp\in \mathcal{M},
  \|xp\|_{\mathcal{M}}\leq\varepsilon, \\
  \ \varphi(z^\bot) \in W(B,\varepsilon,\delta), \
    \mathcal{D}(zp^\bot)\leq\varepsilon \varphi(z)\},
\end{split}
\end{equation}
where $\|\cdot\|_{\mathcal{M}}$ is the $C^*$-norm on $\mathcal{M}$.

It was shown in~\cite{Yead} that the system of sets
\begin{equation*}
 \{x+V(B,\,\varepsilon,\,\delta)\colon \ x \in LS(\mathcal{M}),\
 \varepsilon, \ \delta >0,\ B\in\Sigma,\ \mu(B)<\infty\}
\end{equation*}
defines a Hausdorff vector topology $t(\mathcal{M})$ on
$LS(\mathcal{M})$ such that the sets
$\{x+V(B,\,\varepsilon,\,\delta)\}$, $\varepsilon, \ \delta >0$,
$B\in \Sigma$, $\mu(B)<\infty$ form a neighbourhood base of an
operator $x\in LS(\mathcal{M})$. It is known that
$(LS(\mathcal{M}), t(\mathcal{M}))$ is a complete topological
$*$-algebra, and the topology $t(\mathcal{M})$ does not depend on
a choice of dimension function $\mathcal{D}$~  and on the choice
of $*$-isomorphism $\varphi$ (see e.g. \cite[\S3.5]{MCh},
\cite{Yead}).

The topology $t(\mathcal{M})$ on $LS(\mathcal{M})$ is called the \textit{local
measure topology } (or the \textit{topology of convergence locally in measure}).
Note, that in case when $\mathcal{M}=B(H)$ the equality $LS(\mathcal{M})=\mathcal{M}$
holds \cite[\S 2.3]{MCh} and the topology $t(\mathcal{M})$ coincides with the uniform topology, generated by the $C^*$-norm $\|\cdot\|_{B(H)}$.

We will need the following criterion for convergence of nets with respect to this topology.

\begin{proposition}
\label{p2_1} \cite[\S 3.5]{MCh} (i) A net $\{p_\alpha\}_{\alpha\in
A}\subset \mathcal{P}(\mathcal{M})$ converges to zero with respect
to the topology $t(\mathcal{M})$ if and only if there is a net
$\{z_\alpha\}_{\alpha\in A}\subset
\mathcal{P}(\mathcal{Z}(\mathcal{M}))$ such that $z_\alpha
p_\alpha\in \mathcal{P}_{fin}(\mathcal{M})$ for all $\alpha\in A,\
\varphi(z_\alpha^\bot)\stackrel{t(L^\infty(\Omega))}{\longrightarrow}
0$, and $\mathcal{D}(z_\alpha
p_\alpha)\stackrel{t(L^\infty(\Omega))}{\longrightarrow} 0$, where
$t(L^\infty(\Omega))$ is the local measure topology on
$L^0(\Omega,\Sigma,\mu)$ and $\varphi$ is a $*$-isomorphism of
$\mathcal{Z}(\mathcal{M})$ onto $L^\infty(\Omega,\Sigma,\mu)$.

(ii) A net $\{x_\alpha\}_{\alpha\in A}\subset LS(\mathcal{M})$ converges to zero with respect to the topology $t(\mathcal{M})$ if and only if
$E^\bot_\lambda(|x_\alpha|)\stackrel{t(\mathcal{M})}{\longrightarrow} 0$ for every $\lambda>0$, where $E^\bot_\lambda(|x_\alpha|)$ is a
spectral projection family for the operator $|x_\alpha|$.
\end{proposition}

Since the involution is continuous in the topology
$t(\mathcal{M})$, the set $LS_h(\mathcal{M})$ is closed in
$(LS(\mathcal{M}),t(\mathcal{M}))$. The cone $LS_+(\mathcal{M})$
of positive elements is also closed in
$(LS(\mathcal{M}),t(\mathcal{M}))$ \cite{Yead}.

Using Proposition \ref{p2_1} it is established the following

\begin{proposition}
\label{p2} \cite[Prop.2.3]{BCS3}
If $x_\alpha\in LS(\mathcal{M}),\ 0\neq z \in
\mathcal{P}(\mathcal{Z}(\mathcal{M}))$, then
$$zx_\alpha\stackrel{t(\mathcal{M})}{\longrightarrow} 0\Longleftrightarrow zx_\alpha\stackrel{t(z\mathcal{M})}{\longrightarrow} 0.$$
\end{proposition}

Moreover from Proposition \ref{p2_1} immediately follows

\begin{corollary}
\label{cp2_1} If $\{z_\alpha\}_{\alpha\in A}\subset \mathcal{P}(\mathcal{Z}(\mathcal{M}))$ and $z_\alpha\downarrow 0$ then
$z_\alpha\stackrel{t(\mathcal{M})}{\longrightarrow} 0$.
\end{corollary}

Let us mention the following important property of the topology $t(\mathcal{M})$.

\begin{proposition}
\label{p2_2}
The von Neumann algebra $\mathcal{M}$ is everywhere dense in  $(LS(\mathcal{M}),t(\mathcal{M}))$.
\end{proposition}
\begin{proof}
If $x\in LS(\mathcal{M})$, then there exists a sequence
$\{z_n\}_{n=1}^\infty\subset\mathcal{P}(\mathcal{Z}(\mathcal{M}))$ such that
 $z_n\uparrow \mathbf{1}$ and $xz_n\in S(\mathcal{M})$ for all
$n\in\mathbb{N}$. By Corollary  \ref{cp2_1},
$z_n\stackrel{t(\mathcal{M})}{\longrightarrow} \mathbf{1}$, and therefore
 $xz_n\stackrel{t(\mathcal{M})}{\longrightarrow} x$. Consequently, the algebra
$S(\mathcal{M})$ is every where dense in $(LS(\mathcal{M}),t(\mathcal{M}))$.

Now let $x\in S(\mathcal{M})$. Then there exists a sequence
$\{p_n\}_{n=1}^\infty\subset\mathcal{P}(\mathcal{M})$ such that
$p_n\uparrow \mathbf{1}$, $p_n^\bot\in
\mathcal{P}_{fin}(\mathcal{M})$ and $xp_n\in\mathcal{M}$  for any
$n\in\mathbb{N}$. According to (D7) we have that
$\mathcal{D}(p_n^\bot)\stackrel{t(L^\infty(\Omega))}{\longrightarrow}0$,
therefore, Proposition \ref{p2_1}(i) implies the convergence
$p_n\stackrel{t(\mathcal{M})}{\longrightarrow} \mathbf{1}$ (we set
$z_n=\mathbf{1}$). Then
$xp_n\stackrel{t(\mathcal{M})}{\longrightarrow}x$. It means that
the algebra $\mathcal{M}$ is every where dense in the algebra
$S(\mathcal{M})$ with respect to the topology $t(\mathcal{M})$.
Thus, the von Neumann algebra $\mathcal{M}$ is every where dense
in $(LS(\mathcal{M}),t(\mathcal{M}))$.
\end{proof}

The lattice $\mathcal{P}(\mathcal{M})$ is said to have a countable type if every family of non-zero pairwise orthogonal projections in
$\mathcal{P}(\mathcal{M})$ is, at most, countable.  A von Neumann algebra is said to be $\sigma$-finite if the lattice $\mathcal{P(M)}$ has a countable type.
 It is shown in \cite[Lemma 1.1]{Seg}
that a finite von Neumann algebra  $\mathcal{M}$ is $\sigma$-finite, provided that the
lattice $\mathcal{P}(\mathcal{Z}(\mathcal{M}))$ of central projections has a countable type.

If $\mathcal{M}$ is a commutative von Neumann algebra and
$\mathcal{P}(\mathcal{M})$ has a countable type, then
$\mathcal{M}$ is $*$-isomorphic to the $*$-algebra
$L^\infty(\Omega,\Sigma,\mu)$ with $\mu(\Omega)<\infty$. In this
case, the topology $t(L^\infty(\Omega))$ is metrizable and has a
base of neighbourhoods of zero consisting of the sets
$W(\Omega,1/n,1/n),\ n\in\mathbb{N}$. In addition, $f_n
\stackrel{t(L^\infty(\Omega))}{\longrightarrow}0\Leftrightarrow
f_n\rightarrow 0$ in measure $\mu$, where $f_n,f\in
L^0(\Omega,\Sigma,\mu)=LS(\mathcal{M})$.

We need an another basis of neighbourhoods of zero in topology
$t(\mathcal{M})$ in the case when the algebra
$\mathcal{Z}(\mathcal{M})$ is $\sigma$-finite. If  $\varphi$ is
$*$-isomorphism from $\mathcal{Z}(\mathcal{M})$ onto
$L^\infty(\Omega,\Sigma,\mu),\ \mu(\Omega)<\infty$, then
$\tau(x)=\int_\Omega \varphi(x) d\mu$ is a faithful normal finite
trace on $\mathcal{Z}(\mathcal{M})$. For arbitrary positive
scalars  $\varepsilon,\beta,\gamma$ set
\begin{equation}
\label{v2}
\begin{split}
  V(\varepsilon, \beta,\gamma) := \{x\in LS(\mathcal{M})\colon \
    \mbox{there exist}  \ p\in \mathcal{P}(\mathcal{M}),\ \\
  z\in \mathcal{P}(\mathcal{Z}(\mathcal{M})),\
   \mbox{such that} \ xp\in \mathcal{M},
  \|xp\|_{\mathcal{M}}\leq\varepsilon,
  \\ \tau(z^\bot)\leq \beta, \
    \mathcal{D}(zp^\bot)\leq\gamma \varphi(z)\}.
\end{split}
\end{equation}

\begin{proposition}
\label{p3} If the centre $\mathcal{Z}(\mathcal{M})$ of von Neumann
algebra $\mathcal{M}$ is  $\sigma$-finite algebra, then the system of
sets given by (\ref{v2}) forms a basis of neighbourhoods of zero
in the topology  $t(\mathcal{M})$.
\end{proposition}
\begin{proof}
Let $V(\Omega,\varepsilon,\delta)$ be a neighbourhood of zero of
the form (\ref{v1}). If $x\in V(\varepsilon,\delta,\varepsilon)$,
then there exist
 $p\in\mathcal{P}(\mathcal{M}),\ z\in\mathcal{P}(\mathcal{Z}(\mathcal{M}))$,
such that $xp\in\mathcal{M},\ \|xp\|_{\mathcal{M}}\leq
\varepsilon,\ \int_\Omega \varphi(z^\bot) d \mu\leq \delta$ and
$\mathcal{D}(zp^\bot)\leq\varepsilon\varphi(z)$. The inequality
$\int_\Omega \varphi(z^\bot) d \mu\leq \delta$ means that
$\varphi(z^\bot)\in W(\Omega,\varepsilon,\delta)$. Hence $x\in
V(\Omega,\varepsilon,\delta)$, that implies the inclusion
$V(\varepsilon,\delta,\varepsilon)\subset
V(\Omega,\varepsilon,\delta)$.

If $x\in V(\Omega,\varepsilon,\delta)$, then there exist
$p\in\mathcal{P}(\mathcal{M}),\ z\in
\mathcal{P}(\mathcal{Z}(\mathcal{M}))$, such that
$\|xp\|_{\mathcal{M}}\leq\varepsilon,\ \varphi(z^\bot)\in
W(\Omega,\varepsilon,\delta)$ and
$\mathcal{D}(zp^\bot)\leq\varepsilon \varphi(z)$. Inclusion
$\varphi(z^\bot)\in W(\Omega,\varepsilon,\delta)$ means that there
exists $E\in\Sigma$, such that $\mu(\Omega\setminus E)\leq\delta$
and $0\leq\varphi(z^\bot)\chi_E\leq\varepsilon$. If
$0<\varepsilon<1$, then $\varphi(z^\bot)\chi_E=0$, i.e.
$\varphi(z^\bot)\leq\chi_{\Omega\setminus E}$, and therefore
$\tau(z^\bot)\leq\delta$, that implies $x\in
V(\varepsilon,\delta,\varepsilon)$.
\end{proof}

\section{The self-adjoint derivations on algebra $LS(\mathcal{M})$}

Let $\mathcal{M}$ be an arbitrary von Neumann algebra, let $\mathcal{A}$ be a subalgebra in $LS(\mathcal{M})$. A linear mapping
 $\delta: \mathcal{A} \rightarrow LS(\mathcal{M})$
is called a \emph{derivation}  on $\mathcal{A}$ with values in $LS(\mathcal{M})$, if
$\delta(xy)=\delta(x)y+x\delta(y)$ for all $x,y\in
\mathcal{A}$. Each element $a\in \mathcal{A}$ defines a
derivation  $\delta_a(x):=[a,x]=ax-xa$ on $\mathcal{A}$ with values in $\mathcal{A}$.
Derivations $\delta_a,a\in\mathcal{A}$, are said to be
\emph{inner derivations} on $\mathcal{A}$.

Now, we list a few properties of derivations on $\mathcal{A}$ which we shall need below.

\begin{lemma}
\label{ll1} If $\mathcal{P}(\mathcal{Z}(\mathcal{M}))\subset \mathcal{A}$, $\delta$ is a derivation on $\mathcal{A}$
and  $z\in\mathcal{P}(\mathcal{Z}(\mathcal{M}))$, then $\delta(z)=0$ and
$\delta(zx)=z\delta(x)$ for all $x\in \mathcal{A}$.
\end{lemma}
\begin{proof}
We have that
$\delta(z)=\delta(z^2)=\delta(z)z+z\delta(z)=2z\delta(z)$.
Hence, $z\delta(z)=z(2z\delta(z))=2z\delta(z)$, that is
$z\delta(z)=0$. Therefore, we have $\delta(z)=0$. In particular,
$\delta(zx)=\delta(z)x+z\delta(x)=z\delta(x)$.
\end{proof}

Lemma \ref{ll1} immediately implies

\begin{corollary}
\label{cl1}
If $z\in\mathcal{P}(\mathcal{Z}(\mathcal{M}))\subset \mathcal{A}$, $\delta$ is a derivation on $\mathcal{A}$,
then $\delta(z\mathcal{A})\subset z\mathcal{A}$ and the restriction $\delta^{(z)}$ of the derivation $\delta$ to $z\mathcal{A}$
is a derivation on  $z\mathcal{A}$, in addition, if $\delta$ is $t(\mathcal{M})$-continuous, then
 $\delta^{(z)}$ is $t(z\mathcal{M})$-continuous.
\end{corollary}
\begin{proof}
By Lemma \ref{ll1}, the inclusion $\delta(z\mathcal{A})\subset z\mathcal{A}$ holds.
Moreover, the linear mapping $\delta^{(z)}: z\mathcal{A}\longrightarrow z\mathcal{A}$
has the following property
$$
\delta^{(z)}((zx)(zy))=\delta(zx)zy+zx\delta(zy)=\delta^{(z)}(zx)zy+zx\delta^{(z)}(zy)
$$
for all $x,y\in\mathcal{A}$.

If $x_\alpha,x\in z\mathcal{A}$,
$x_\alpha\stackrel{t(z\mathcal{M})}{\longrightarrow} x$, then
$x_\alpha\stackrel{t(\mathcal{M})}{\longrightarrow} x$
(Proposition \ref{p2}), and therefore
$\delta^{(z)}(x_\alpha)=z\delta(x_\alpha)\stackrel{t(\mathcal{M})}{\longrightarrow}
z\delta(x)=\delta^{(z)}(x)$, that implies the convergence
$\delta^{(z)}(x_\alpha)\stackrel{t(z\mathcal{M})}{\longrightarrow}
\delta^{(z)}(x)$ (Proposition \ref{p2}).
\end{proof}

Let $\mathcal{A}$ be a subalgebra in $LS(\mathcal{M}),\ 0\neq e\in\mathcal{P}(\mathcal{M})\cap\mathcal{A}$, let
$\delta$ be a derivation on $\mathcal{A}$ and let $\delta^{(e)}$ be a linear mapping from  $e\mathcal{A}e$ into $e\mathcal{A}e$ defined by the equality
 $\delta^{(e)}(x)=e\delta(x)e,\ x\in e\mathcal{A}e\subset\mathcal{A}$.
If $e=z\in\mathcal{P}(\mathcal{Z}(\mathcal{M}))$, then $\delta^{(e)}$ coincides with the restriction  $\delta^{(z)}$ of the derivation
 $\delta$ to $z\mathcal{A}=z\mathcal{A}z$.

\begin{lemma}
\label{ll2}
$\delta^{(e)}$ is a derivation on $e\mathcal{A}e$, in addition, $\delta^{(e)}(e)=0$.
\end{lemma}
\begin{proof}
If $x,y\in e\mathcal{A}e$, then $x,y\in \mathcal{A}$ and
\begin{equation*}
 \begin{split}
\delta^{(e)}(xy)&=
e(\delta(x)y)e+e(x\delta(y))e\cr &
=(e\delta(x)e)(eye)+(exe)(e\delta(y)e)=
\delta^{(e)}(x)y+x\delta^{(e)}(y).
 \end{split}
\end{equation*}
Further, from the equalities
$$
\delta^{(e)}(e)=e\delta(e^2)e=e\delta(e)ee+ee\delta(e)e=2e\delta(e)e=2\delta^{(e)}(e)
$$
it follows that $\delta^{(e)}(e)=0$.
\end{proof}

Let $\mathcal{A}$ be a $*$-subalgebra in $LS(\mathcal{M})$, let $\delta$ be a derivation on $\mathcal{A}$ with values in $LS(\mathcal{M})$. Let us define a mapping
$$\delta^*: \mathcal{A}\rightarrow
LS(\mathcal{M}),$$
by setting $\delta^*(x)=(\delta(x^*))^*$, $x\in
\mathcal{A}$. A direct verification shows that $\delta^*$ is also a
derivation on $\mathcal{A}$.

A derivation $\delta$ on $\mathcal{A}$ is said to be
\emph{self-adjoint}, if $\delta=\delta^*$. Every
derivation $\delta$ on $\mathcal{A}$ can be represented in the form
$\delta= Re(\delta)+ i Im(\delta)$, where
$Re(\delta)=(\delta+\delta^*)/2,\ Im(\delta)=(\delta-\delta^*)/2i$
are self-adjoint derivations on $\mathcal{A}$.

Since $(LS(\mathcal{M}),t(\mathcal{M}))$ is a topological
 $*$-algebra, the following result holds.

\begin{lemma}
\label{ll3} If $\mathcal{A}$ is a $*$-subalgebra in $LS(\mathcal{M})$, then a derivation $\delta: \mathcal{A}\rightarrow
LS(\mathcal{M})$ is continuous with respect to the topology  $t(\mathcal{M})$ if and only if the self-adjoint derivations $Re(\delta)$
and $Im(\delta)$ are continuous with respect to this topology.
\end{lemma}

The following lemma establishes a connection between the property of innerness of  derivation $\delta$ and the property of innerness of derivations
$Re(\delta)$ and $Im(\delta)$.

\begin{lemma}
\label{ll4} Let $\mathcal{A}$ be a $*$-subalgebra in $LS(\mathcal{M})$. A derivation $\delta: \mathcal{A}\rightarrow
\mathcal{A}$ is inner derivation on $\mathcal{A}$ if and only if $Re(\delta)$ and $Im(\delta)$
are inner derivations.
\end{lemma}
\begin{proof}
If $\delta=\delta_a,\ a\in\mathcal{A},\ b=Re(a)=(a+a^*)/2,\
c=Im(a)=(a-a^*)/2i,\ x\in\mathcal{A}$, then

\begin{gather*}
\begin{split}
Re(\delta)(x)&=\frac{[a,x]+[a,x^*]^*}{2}=\frac{ax-xa+(ax^*-x^*a)^*}{2}\\
&=icx-ixc=\delta_{ic}(x);
\end{split}
\end{gather*}
and
\begin{gather*}
\begin{split}
Im(\delta)(x)&=\frac{[a,x]-[a,x^*]^*}{2i}=\frac{ax-xa-(ax^*-x^*a)^*}{2i}\\
&=-ibx+ixb=\delta_{-ib}(x).
\end{split}
\end{gather*}

Conversely, if $Re(\delta)=\delta_a,\ Im(\delta)=\delta_b$ for some $a,b\in\mathcal{A}$, then for all $x\in\mathcal{A}$ the equality
$$\delta(x)=Re(\delta)(x)+i Im(\delta)(x)=[a,x]+i[b,x]=[a+ib,x]=\delta_{a+ib}(x)$$
holds.
\end{proof}

\begin{lemma}
\label{ll5} Let $\mathcal{A}$ and $\delta$ be the same as in
Corollary \ref{cl1} and let $\{z_i\}_{i\in I}$ be a central
decomposition of the unity $\mathbf{1}$. If
$\delta^{(z_i)}=\delta_{d_i},\ d_i\in z_i\mathcal{A}$ is an inner
derivation on  $z_i\mathcal{A}$ for every $i\in I$, then there
exists an operator $d\in LS(\mathcal{M})$, such that
$\delta(x)=[d,x]$ for all $x\in\mathcal{A}$ and $z_i d=d_i$ for
every $i\in I$.
\end{lemma}
\begin{proof}
Since $\{z_i\}_{i\in I}$ is a central decomposition
of the unity $\mathbf{1}$, $d_i\in z_i\mathcal{A}\subset z_i LS(\mathcal{M})$,  by Proposition \ref{p1} there exists $d\in LS(\mathcal{M})$, such that $z_i d=d_i$ for every $i\in I$. Using Lemma \ref{ll1} and equalities
$\delta^{(z_i)}(x)=[d_i,x],\ x\in z_i\mathcal{A},\ i\in I$ we have that for all $y\in\mathcal{A},\ i\in I$ equalities $z_i\delta(y)=\delta^{(z_i)}(z_i y)=[d_i,z_i y]=[z_id,z_i y]=z_i[d,y]$ hold. Since  $\sup_{i\in I}z_i=\mathbf{1}$ it follows that $\delta(y)=[d,y]$.
\end{proof}

\begin{lemma}
\label{ll6} Let $\delta$ be a derivation on a subalgebra $\mathcal{A}$ of $LS(\mathcal{M})$ and $\mathcal{P}(\mathcal{M})\subset\mathcal{A}$.
If $p,q\in\mathcal{P}(\mathcal{M})$ and $p\delta(q)p\geq \lambda p$ for some $\lambda>0$, then $$ r(qp)\delta(l(qp))r(qp)\geq\lambda r(qp).$$
\end{lemma}
\begin{proof}
Set $e=l(qp)$ and $f=r(qp)$. It is clear that $eq=e$ and $pf=f$. In addition, $e=r((qp)^*)=r(pq)$ and $f=l((qp)^*)=l(pq)$.

Since $$ef=(eq)(pf)=e(qp)f=l(qp)qpr(qp)=qp=(qp)f=q(pf)=qf$$ and $$fe=f(pq)e=pq=f(pq)=(fp)q=fq,$$ we have

\begin{gather*}
\begin{split}
f\delta(e)f&=f(f\delta(e))f=f\delta(fe)f-f(\delta(f)e)f\\
&=f\delta(fq)f-f\delta(f)qf=
 f\delta(f)qf+f(f\delta(q))f-f\delta(f)qf\\
&=f\delta(q)f=f(p\delta(q)p)f\geq
\lambda fpf=\lambda f.
\end{split}
\end{gather*}
\end{proof}

For every $x\in LS(\mathcal{M})$ we set $s(x)=l(x)\vee r(x)$, where $l(x)$ and $r(x)$ are left and right supports of $x$ respectively.

Let $\mathcal{M}_{h,1}=\{x\in\mathcal{M}_h:\ \|x\|_{\mathcal{M}}\leq 1\}$. Fix a positive number $\lambda$ and a self-adjoint derivation
$\delta: LS(\mathcal{M})\longrightarrow LS(\mathcal{M})$.
The set of pairs $\mathcal{S}=\{(p_j,x_j)\in\mathcal{P}(\mathcal{M})\times\mathcal{M}_{h,1}:\ p_j\neq 0,\ j\in J\}$
is called \emph{$\lambda$-system }
(for the derivation $\delta$), if

(i). $(p_j\vee s(x_j))(p_i\vee s(x_i))=0$ and $(p_j\vee s(x_j))s(\delta(p_i))=0$ for $j\neq i$, $j,i\in J$;

(ii). $s(x_j)\sim p_j$ for all $j\in J$;

(iii). $p_j\delta(x_j)p_j\geq \lambda p_j$ for all $j\in J$.

The projection $\bigvee_{j\in J}(p_j\vee s(x_j)\vee s(\delta(p_j))\vee s(\delta(p_j\vee s(x_j))))$ is called the \emph{support} of the $\lambda$-system
$\mathcal{S}$ and is denoted by $s(\mathcal{S})$. If $\lambda$-system $\mathcal{S}$ is empty set, then we set $s(\mathcal{S})=0$.

A $\lambda$-system is said to be \emph{maximal} if it does not contained in any larger $\lambda$-system.

\begin{theorem}
\label{tt1}
Let $\mathcal{S}=\{(p_j,x_j)\}_{j\in J}$ be a maximal $\lambda$-system of a self-adjoint derivation
$\delta: LS(\mathcal{M})\longrightarrow LS(\mathcal{M})$, $g=s(\mathcal{S})^\bot$ and $\delta^{(g)}(x)=g\delta(x)g,\ x\in g LS(\mathcal{M})g$. Then
\begin{equation}
\label{tt1v}
\delta^{(g)}(g\mathcal{M}g)\subset g\mathcal{M}g.
\end{equation}
\end{theorem}
\begin{proof}
Let us first prove that
\begin{equation}
\label{tt1_1}
\delta^{(g)}(q)\subset g\mathcal{M}g\ \text{and}\ \|\delta^{(g)}(q)\|_{\mathcal{M}}\leq \lambda
\end{equation}
for any projection $q\in \mathcal{P}(g\mathcal{M}g)$.
Since $\delta^*=\delta$, it follows that $\delta(q)\in LS_h(\mathcal{M})$ and therefore $\delta^{(g)}(q)\in LS_h(g\mathcal{M}g)$.
Let $\delta^{(g)}(q)\neq 0$ and let $p$ be the spectral projection for  $\delta^{(g)}(q)$ corresponding to the interval $(\lambda,+\infty)$. It is clear that
$p\leq s(\delta^{(g)}(q))\leq g$.

Suppose that $p\neq 0$, then
\begin{equation}
\label{tt1_2}
0\neq \lambda p\leq p\delta^{(g)}(q)p=p\delta(q)p.
\end{equation}
Since $$0\neq
p\delta(q)p=\delta(pq)p-\delta(p)qp=\delta(pq)p-\delta(p)(pq)^*,$$
it follows that $qp=(pq)^*\neq 0$. Consequently, $$e=l(qp)\neq 0\ \text{and}\
f=r(qp)\neq 0,$$ in addition, $e\sim f$. Since
$g=s(\mathcal{S})^\bot$, from the inequalities $f\leq p\leq g$ and
$e\leq q\leq g$ it follows that $$(f\vee e)(p_j\vee s(x_j)))=0,\
(f\vee e)s(\delta(p_j))=0$$ and $$(p_j\vee
s(x_j))\delta(f)=\delta((p_j\vee s(x_j))f)-\delta(p_j\vee
s(x_j))f=0$$ for all $j\in J$. Moreover, according to (\ref{tt1_2})
and Lemma \ref{ll6} we have that $f\delta(e)f\geq \lambda f$. Thus, the system $\mathcal{S}\cup \{(f,e)\}$ is a
$\lambda$-system, that contradicts to maximality of the
$\lambda$-system $\mathcal{S}$. Consequently, $p=0$, which implies the inequality $\delta^{(g)}(q)\leq\lambda\mathbf{1}$. Similarly, for projection $(g-q)\leq g$ we obtain that
$$g(\delta(g-q))g=\delta^{(g)}(g-q)\leq\lambda\mathbf{1}.$$
By Lemma \ref{ll2}, $g\delta(g)g=0$, and therefore $-g\delta(q)g\leq\lambda\mathbf{1}$. Thus,
$$-\lambda\mathbf{1}\leq g\delta(q)g\leq\lambda\mathbf{1},$$
i.e. $\delta^{(g)}(q)\in g\mathcal{M}g$ and $\|\delta^{(g)}(q)\|_{\mathcal{M}}\leq \lambda$.

Now, suppose that the inclusion (\ref{tt1v}) false, i.e. there exists an element $x\in\mathcal{M}_{h,1}\cap(g\mathcal{M}g)$,
such that $\delta^{(g)}(x)\in LS_h(\mathcal{M})\setminus g\mathcal{M}g$. It means that the spectral projection $r=E^\bot_{3\lambda}(\delta^{(g)}(x))$
(or $r=E_{-3\lambda}(\delta^{(g)}(x))$)
for $\delta^{(g)}(x)$ corresponding to the interval $(3\lambda,+\infty)$ (respectively, $(-\infty,-3\lambda$), is not equal to zero.
Replacing, if necessary,  $x$ on $-x$, we may assume that  $r=E^\bot_{3\lambda}(\delta^{(g)}(x))\neq 0$.
It is clear that $r\leq s(\delta^{(g)}(x))\leq g$ and
\begin{equation}
\label{tt1_3}
0<3\lambda r\leq r\delta^{(g)}(x)r=r\delta(x)r.
\end{equation}
According to (\ref{tt1_1}) we have that $\|\delta^{(g)}(r)\|_{\mathcal{M}}\leq\lambda$, and therefore inclusion
$x\in\mathcal{M}_{h,1}\cap g\mathcal{M}g$ and equality $$r\delta(r)xr+rx\delta(r)r=rg\delta(r)gxr+rxg\delta(r)gr$$
imply that $$\|r\delta(r)xr+rx\delta(r)r\|_{\mathcal{M}}\leq 2\lambda.$$
Consequently,
\begin{equation}
\label{tt1_4}
-2\lambda r\leq r\delta(r)xr+rx\delta(r)r\leq 2\lambda r.
\end{equation}
Using (\ref{tt1_3}) and (\ref{tt1_4}) for $y=rxr$ we obtain that
\begin{equation}
\label{tt1_5}
r\delta(y)r=r\delta(rxr)r=r\delta(r)xr+r\delta(x)r+rx\delta(r)r\geq \lambda r>0,
\end{equation}
in particular, $y\neq 0$ and $q=s(y)\neq 0$.
Let us show that the collection $\mathcal{S}\cup \{(q,y)\}$ forms a $\lambda$-system.
Since $q\leq r\leq g$, from (\ref{tt1_5}) it follows that $q\delta(y)q\geq \lambda q$, in addition
$$(q\vee s(y))(p_j\vee s(x_j))=0=q(p_j\vee s(x_j)),\ q\delta(p_j)=0$$ and
$$(p_j\vee s(x_j))\delta(q)=\delta((p_j\vee s(x_j))q)-\delta(p_j\vee s(x_j))q=0$$ for all $j\in J$.
It means that the set $\mathcal{S}\cup \{(q,y)\}$ is a $\lambda$-system, that contradicts to maximality of the $\lambda$-system $\mathcal{S}$. From obtained contradiction follows the validity of inclusion  (\ref{tt1v}).
\end{proof}

\begin{lemma}
\label{ll7} If $\{x_j\}_{j\in J}\subset\mathcal{M}_{h,1},\ x_ix_j=0,\ i\neq j,\ i,j\in J$, then there exists a unique element $x\in\mathcal{M}_{h,1}$, denoted by $\sum_{j\in J}x_j$, such that $xs(x_j)=x_j$ for all $j\in J$ and $\sup_{j\in J}s(x_j)=s(x)$.
\end{lemma}
\begin{proof}
Denote by $\mathcal{A}$ the commutative von Neumann subalgebra of $\mathcal{M}$, containing the family $\{x_j\}_{j\in J}$.
Since $\mathcal{A}_h$ is an order complete vector lattice,
$\{x_j\}_{j\in J}$ is the family of pairwise disjoint element of $\mathcal{A}_h$ and
$|x_j|\leq\mathbf{1}\in\mathcal{A}$ for all $j\in J$, it follows that there exists a unique element $x$ of $\mathcal{A}_h$ such that $|x|\leq\mathbf{1}$, $xs(x_j)=x_j$ and $s(x)=\sup_{j\in J}s(x_j)$.

Let $y$ be another element of $\mathcal{M}_{h,1}$, such that $ys(x_j)=x_j$ for all $j\in J$ and $\sup_{j\in J}s(x_j)=s(y)$.
Then $(x-y)s(x_j)=x_j-x_j=0$ for any $j\in J$. Therefore $s(x)=\sup_{j\in J}s(x_j)\leq (r(x-y))^\bot$ and then $x-y=xs(x)-ys(y)=xs(x)-ys(x)=(x-y)s(x)=0$.
\end{proof}

\begin{lemma}
\label{ll8} Let $x\in LS_h(\mathcal{M}),\ p,q\in\mathcal{P}(\mathcal{M}),\ \rho,\lambda\in\mathbb{R},\ \rho<\lambda$,
\begin{equation}
\label{ll8_1}
pxp\leq \rho p
\end{equation}
and
\begin{equation}
\label{ll8_2}
qxq\geq \lambda q.
\end{equation}
Then $p\preceq q^\bot$ and $q\preceq p^\bot$.
\end{lemma}
\begin{proof}
Set $r=p\wedge q$. Multiplying both parts on both sides of inequalities (\ref{ll8_1}) and (\ref{ll8_2}) by $r$, we obtain that
$$\lambda r\leq rxr\leq \rho r,$$
that is possible if $r=0$ only. Therefore $p=p-p\wedge q\sim p\vee q -q\leq q^\bot$, i.e. $p\preceq q^\bot$. Similarly, $q\preceq p^\bot$.
\end{proof}

\begin{theorem}
\label{p4} Let $\mathcal{S}=\{(p_j,x_j)\}_{j\in J}$ be a
$\lambda$-system for a self-adjoint derivation $\delta:
LS(\mathcal{M})\longrightarrow LS(\mathcal{M})$, let $\mathcal{D}$ be a dimensional function on $\mathcal{P}(\mathcal{M})$. Then
\begin{equation}
\label{p4v}
\mathcal{D}(s(\mathcal{S}))\leq 8\mathcal{D}(E^\bot_{\rho}(\delta(\sum_{j\in J}x_j)))\ \text{for any}\ \rho<\lambda.
\end{equation}
\end{theorem}
\begin{proof}
Set $x=\sum_{j\in J}x_j$ (see Lemma \ref{ll7}) and $p=\sup_{j\in J}p_j$. Let us show that
\begin{equation}
\label{p4_1}
p\delta(x)p\geq\lambda p.
\end{equation}
Since $(p_j\vee s(x_j))(p_i\vee s(x_i))=0$ and $(p_j\vee
s(x_j))s(\delta(p_i))=0$ for $i\neq j$, it follows that
$x_ip_j=x_is(x_i)p_j=0$ and
$x_i\delta(p_j)=x_is(x_i)s(\delta(p_j))\delta(p_j)=0$ for $i\neq
j$. Therefore $$\delta(x_i)p_j=\delta(x_ip_j)-x_i\delta(p_j)=0,$$
that implies equality $$s(\delta(x_i))p_j=0\ \text{for}\ i\neq
j.$$ From here and from the equality $p=\sup_{j\in J}p_j$ it
follows that $$s(\delta(x_i))p=s(\delta(x_i))p_i.$$ Thus,
\begin{equation}
\label{p4_1_1}
\delta(x_i)p=\delta(x_i)s(\delta(x_i))p=\delta(x_i)p_i.
\end{equation}
By Lemma \ref{ll7}, we have that
\begin{equation}
\label{p4_1_2}
p_ix=p_is(x)x=(p_i\sup_{j\in J}s(x_j))x=p_is(x_i)x=p_ix_i
\end{equation}
and
\begin{equation}
\label{p4_1_2_1}
x_ip=x_i(s(x_i)\sup_{j\in J}p_j)=x_i(s(x_i)p_i)=x_ip_i.
\end{equation}
Similarly,
\begin{equation}
\label{p4_1_3}
\begin{split}
\delta(p_i)xp&=\delta(p_i)(s(\delta(p_i))\sup_{j\in J}s(x_j))xp\\&=
\delta(p_i)s(x_i)xp=\delta(p_i)x_ip=\delta(p_i)x_ip_i.
\end{split}
\end{equation}
By (\ref{p4_1_1}), (\ref{p4_1_2}) and (\ref{p4_1_2_1}), we obtain
\begin{gather*}
\begin{split}
\delta(p_ix)p&=\delta(p_ix_i)p=\delta(p_i)x_ip+p_i\delta(x_i)p\\
&= \delta(p_i)x_ip_i+p_i\delta(x_i)p_i=\delta(p_ix_i)p_i,
\end{split}
\end{gather*}
that according to (\ref{p4_1_3}) implies equalities
\begin{gather*}
\begin{split}
p_i(p\delta(x)p)&=p_i\delta(x)p=\delta(p_ix)p-\delta(p_i)xp\\
&= \delta(p_ix_i)p_i-\delta(p_i)x_ip_i=p_i\delta(x_i)p_i.
\end{split}
\end{gather*}
Hence,
\begin{equation}
\label{p4_1_4}
p_i(p\delta(x)p)=p_i\delta(x_i)p_i,
\end{equation}
in particular, the projection  $p_i$ commutates with the operator
$p\delta(x)p$. Set $y=p\delta(x)p-\lambda p$ and by
$y_-=(|y|-y)/2$ denote the negative part of the operator $y$.
Since $p_iy=yp_i$ (see (\ref{p4_1_4})) and
$p_i\delta(x_i)p_i\geq\lambda p_i$ (see the definition of
$\lambda$-system), it follows that
\begin{equation}
\label{p4_1_5}
y_-p_i=p_iy_-=(p_i(p\delta(x)p-\lambda p))_-\stackrel{(\ref{p4_1_4})}{=}(p_i\delta(x_i)p_i-\lambda p_i)_-=0
\end{equation}
for all $i\in J$. From equalities ($\ref{p4_1_5}$) and $p=\sup_{j\in J}p_j$ according to \cite[Prop. 2.4.1(ix)]{MCh}, it follows that
\begin{equation}
\label{p4_1_5_1} (pyp)_-=p(p\delta(x)p-\lambda p)_-p=py_-p=0.
\end{equation}
Therefore
$$pyp=(pyp)_+-(pyp)_-\stackrel{(\ref{p4_1_5_1})}{=}(pyp)_+\geq 0,$$
that implies inequality (\ref{p4_1}).

Fix a real number  $\rho<\lambda$ and set
$q=E_{\rho}(\delta(x))$. By Lemma \ref{ll8}, we obtain
\begin{equation}
\label{p4_3}
p\preceq q^\bot.
\end{equation}
For every fixed $j\in J$ we have that
$$\delta(p_j)=\delta(p_j^2)=\delta(p_j)p_j+p_j\delta(p_j)=\delta(p_j)p_j+(\delta(p_j)p_j)^*$$
and therefore
$$s(\delta(p_j))\leq l(\delta(p_j)p_j)\vee p_j,$$
that implies
\begin{equation}
\label{p4_4}
\mathcal{D}(s(\delta(p_j)))\leq \mathcal{D}(l(\delta(p_j)p_j))+\mathcal{D}(p_j).
\end{equation}
Since $l(\delta(p_j)p_j)\sim r(\delta(p_j)p_j)\leq p_j$, according to (\ref{p4_4}) we have
\begin{equation}
\label{p4_5}
\mathcal{D}(s(\delta(p_j)))\leq 2\mathcal{D}(p_j)
\end{equation}
for all $j\in J$. Similarly,
$$\mathcal{D}(s(\delta(p_j\vee s(x_j))))\leq 2\mathcal{D}(p_j\vee s(x_j)),$$
and in view of equivalence $p_j\sim s(x_j)$ (see the definition of $\lambda$-system) we obtain
\begin{equation}
\label{p4_6}
\mathcal{D}(s(\delta(p_j\vee s(x_j))))\leq 4\mathcal{D}(p_j).
\end{equation}
Denote by $A$ the directed set of all finite subsets of $J$ ordered by inclusion and for every $\alpha\in A$ set
$$e_\alpha=\bigvee_{j\in\alpha}(p_j\vee s(x_j)\vee s(\delta(p_j))\vee s(\delta(p_j\vee s(x_j)))).$$
From properties (D2), (D3) of the dimensional function $\mathcal{D}$ and from inequalities (\ref{p4_3}), (\ref{p4_5}) and (\ref{p4_6}) we have that
\begin{gather*}
\begin{split}
\mathcal{D}(e_\alpha)&\leq\sum_{j\in\alpha}\mathcal{D}(p_j\vee s(x_j)\vee s(\delta(p_j))\vee s(\delta(p_j\vee s(x_j))))\\
&\leq
\sum_{j\in\alpha}(\mathcal{D}(p_j)+\mathcal{D}(s(x_j))+\mathcal{D}(s(\delta(p_j)))+\mathcal{D}(s(\delta(p_j\vee s(x_j)))))\\
&\leq
8\sum_{j\in\alpha}\mathcal{D}(p_j)=8\mathcal{D}(\sum_{j\in\alpha}p_j)\leq 8\mathcal{D}(p)\leq 8\mathcal{D}(q^\bot).
\end{split}
\end{gather*}
Since $e_\alpha\uparrow s(\mathcal{S})$ the last inequality and property (D6) of the dimensional function $\mathcal{D}$ imply that
$$\mathcal{D}(s(\mathcal{S}))=\mathcal{D}(sup_{\alpha\in A}e_\alpha)=sup_{\alpha\in A}\mathcal{D}(e_\alpha)\leq 8\mathcal{D}(q^\bot).$$
\end{proof}

\section{Automatic innerness of continuous derivations on algebra $LS(\mathcal{M})$}

Let $\mathcal{M}$ be an arbitrary von Neumann algebra and let
$\delta_a(x)=[a,x],\\ a,x\in LS(\mathcal{M})$ be an inner derivation on $LS(\mathcal{M})$. Since
$(LS(\mathcal{M}),t(\mathcal{M}))$ is a topological algebra, every derivation
 $\delta_a$ is continuous with respect to the topology
$t(\mathcal{M})$.

The main result of this section is the proof of inverse implication.

\begin{theorem}
\label{t1}
Every derivation on the algebra $LS(\mathcal{M})$ continuous with respect to the topology $t(\mathcal{M})$ is inner derivation.
\end{theorem}
\begin{proof}
Let $\delta$ be an arbitrary derivation on the $*$-algebra $LS(\mathcal{M})$ and let $\delta$ be continuous with respect to the topology  $t(\mathcal{M})$. By Lemmas  \ref{ll3} and \ref{ll4}, we may assume that $\delta$ is a self-adjoint derivation.

Choose a central decomposition $\{z_i\}_{i\in I}$ of the unity
$\mathbf{1}$, such that every Boolean algebra
$z_i\mathcal{P}(\mathcal{Z}(\mathcal{M}))$ has a countable type,
$i\in I$. By Corollary \ref{cl1} the restriction $\delta^{(z_i)}$
of the derivation $\delta$ to $z_i
LS(\mathcal{M})=LS(z_i\mathcal{M})$ is a
$t(z_i\mathcal{M})$-continuous derivation on
$LS(z_i\mathcal{M})$. If every derivation $\delta^{(z_i)},\
i\in I$ is inner, then, by Lemma \ref{ll5}, the derivation $\delta$ is inner too. Thus, in the proof of Theorem \ref{t1} we may assume that the centre
$\mathcal{Z}(\mathcal{M})$ of von Neumann algebra $\mathcal{M}$ is
$\sigma$-finite algebra. In this case, there exist a faithful normal finite trace $\tau(x)=\int\varphi(x) d\mu$ on $\mathcal{Z}(\mathcal{M})$ and the vector topology
$t(\mathcal{M})$ has the basis of neighbourhoods of zero consists of the sets
$V(\varepsilon,\beta,\gamma)$ given by (\ref{v2}) (see
Proposition \ref{p3}). Since the derivation $\delta$ is
$t(\mathcal{M})$-continuous, for arbitrary
$\varepsilon,\beta,\gamma>0$ there exist
$\varepsilon_1,\beta_1,\gamma_1>0$, such that
$\delta(V(\varepsilon_1,\beta_1,\gamma_1))\subset
V(\varepsilon,\beta,\gamma)$. It is clear that
$$\mathcal{M}_1:=\{x\in\mathcal{M}:\ \|x\|_{\mathcal{M}}\leq
1\}\subset
V(1,\beta_1,\gamma_1)=\varepsilon_1^{-1}V(\varepsilon_1,\beta_1,\gamma_1),$$
and therefore $$\delta(\mathcal{M}_1)\subset
\varepsilon_1^{-1}V(\varepsilon,\beta,\gamma)=V(\varepsilon/\varepsilon_1,\beta,\gamma).$$
Hence, for $t(\mathcal{M})$-continuous self-adjoint derivation $\delta: LS(\mathcal{M})\rightarrow LS(\mathcal{M})$ and for arbitrary positive numbers  $\beta$ and
$\gamma$ there exists a number $\Delta(\beta,\gamma)$, such that
\begin{equation}
\label{t1_1}
\delta(\mathcal{M}_1)\subset V(\Delta(\beta,\gamma),\beta,\gamma).
\end{equation}
Let $\mathcal{D},\varphi,\tau$ be the same as in the definition of the set $V(\varepsilon,\beta,\gamma)$ from (\ref{v2}).
Take an arbitrary $2\Delta(\beta,\gamma)$-system $\mathcal{S}=\{(p_j,x_j)\}_{j\in J}$ for the derivation $\delta$
and show that there exists a central projection  $z\in\mathcal{P}(\mathcal{Z}(\mathcal{M}))$, such that
\begin{equation}
\label{t1_2}
\tau(z^\bot)\leq \beta\ \text{and}\ \mathcal{D}(z s(\mathcal{S}))\leq 8\gamma\varphi(z).
\end{equation}
If $\mathcal{S}$ is empty set, then $s(\mathcal{S})=0$ and, in this case, relations (\ref{t1_2}) hold for $z=\mathbf{1}$.
Now, let  $\mathcal{S}=\{(p_j,x_j)\}_{j\in J}$ is non empty $2\Delta(\beta,\gamma)$-system.
By Lemma \ref{ll7} there exists $x=\sum_{j\in J}x_j\in\mathcal{M}_{h,1}$. From (\ref{t1_1}) it follows that
$\delta(x)\in V(\Delta(\beta,\gamma),\beta,\gamma)$ for all $\beta,\gamma>0$. Therefore there exist projections
$z\in\mathcal{P}(\mathcal{Z}(\mathcal{M}))$ and $q\in\mathcal{P}(\mathcal{M})$, such that
\begin{equation}
\label{t1_2_1}
\begin{split}
\tau(z^\bot)\leq \beta,\ \delta(x)q\in\mathcal{M},\
\|\delta(x)q\|_{\mathcal{M}}\leq \Delta(\beta,\gamma)\\
\text{and}\ \mathcal{D}(zq^\bot)\leq \gamma \varphi(z).
\end{split}
\end{equation}
Since $x=x^*$ and $\delta=\delta^*$, it follows that $\delta(x)=(\delta(x))^*$ and, according to (\ref{t1_2_1}), we have
\begin{equation}
\label{t1_2_2}
-\Delta(\beta,\gamma)q\leq q\delta(x)q\leq \Delta(\beta,\gamma)q.
\end{equation}
Set $\rho=\frac{3}{2}\cdot\Delta(\beta,\gamma)$. Using inequalities  (\ref{t1_2_2}) and
$$\rho E^\bot_{\rho}(\delta(x))\leq E^\bot_{\rho}(\delta(x))\delta(x)E^\bot_{\rho}(\delta(x)),$$
we obtain that $E^\bot_{\rho}(\delta(x))\preceq q^\bot$ (Lemma
\ref{ll8}). Consequently, $zE^\bot_{\rho}(\delta(x))\preceq
zq^\bot$ and, by (\ref{p4v}) and (\ref{t1_2_1}), we have that
\begin{gather*}
\begin{split}
\mathcal{D}(zs(\mathcal{S}))&\stackrel{(D4)}{=}\varphi(z)\mathcal{D}(s(\mathcal{S}))\stackrel{(\ref{p4v})}{\leq}
8\varphi(z)\mathcal{D}(E^\bot_{\rho}(\delta(x)))\\
&\stackrel{(D4)}{=}
8\mathcal{D}(zE^\bot_{\rho}(\delta(x)))\stackrel{(D2),(D3)}{\leq}
8\mathcal{D}(zq^\bot)\stackrel{(\ref{t1_2_1})}{\leq} 8\gamma \varphi(z),
\end{split}
\end{gather*}
i.e. (\ref{t1_2}) holds.

For every $n\in\mathbb{N}$ choose a maximal (possible, empty)
$2\Delta(2^{-n},2^{-n})$-system $\mathcal{S}_n$ for the derivation $\delta$.
Set $q'_n=s(\mathcal{S}_n)^\bot$. By Theorem \ref{tt1}, we have that
\begin{equation}
\label{t1_2_3}
\delta^{(q'_n)}(q'_n\mathcal{M}q'_n)\subset q'_n\mathcal{M}q'_n
\end{equation}
for all $n\in\mathbb{N}$. Moreover, in view of  (\ref{t1_2}),
there exists a projection
$z'_n\in\mathcal{P}(\mathcal{Z}(\mathcal{M}))$, such that
\begin{equation}
\label{t1_2_4}
\tau(z_{n}^{'\bot})\leq 2^{-n}\ \text{and}\ \mathcal{D}(z'_n q_n^{'\bot})\leq 2^{-n+3} \varphi(z'_n).
\end{equation}
It is clear that the sequences of projections $q_n=\bigwedge_{k=n+1}^\infty q'_k$ and $z_n=\bigwedge_{k=n+1}^\infty z'_k$ are increasing, in addition
\begin{equation}
\label{t1_2_5}
\tau(z_n^\bot)\leq\tau(\sup_{k\geq n+1}z_{n}^{'\bot})\leq\sum_{k\geq n+1}\tau(z_{n}^{'\bot})\stackrel{(\ref{t1_2_4})}{\leq}\sum_{k\geq n+1}
2^{-k}=2^{-n}
\end{equation}
and
\begin{equation}
\label{t1_2_6}
\begin{split}
\mathcal{D}(z_nq_n^\bot)&=\varphi(z_n)\mathcal{D}(\sup_{k\geq n+1}z_nq_k^{'\bot}) \cr &\leq\varphi(z_n)\mathcal{D}(\sup_{k\geq n+1}z'_kq_k^{'\bot})
\cr
&\stackrel{(D6)}{\leq}
\varphi(z_n)\sum_{k\geq n+1}\mathcal{D}(z'_kq_k^{'\bot})\cr
&\stackrel{(\ref{t1_2_4})}{\leq}
\varphi(z_n)\sum_{k\geq n+1}2^{-k+3}\varphi(z'_k)\cr
&=
\sum_{k\geq n+1}2^{-k+3}\varphi(z_nz'_k)\cr
&=
\sum_{k\geq n+1}2^{-k+3}\varphi(z_n)=2^{-n+3}\varphi(z_n).
\end{split}
\end{equation}

Consider the derivation $\delta^{(q_n)}$ on $q_n LS(\mathcal{M})q_n$ and show that
$$\delta^{(q_n)}(q_n\mathcal{M}q_n)\subset q_n\mathcal{M}q_n.$$

If $x\in q_n\mathcal{M}q_n$, then $x\in
q'_{n+1}\mathcal{M}q'_{n+1}$ and therefore, by (\ref{t1_2_3}),
$$\delta^{(q_n)}(x)=q_n\delta(x)q_n=q_n q'_{n+1} \delta(x)q'_{n+1}
q_n=$$ $$q_n\delta^{(q'_{n+1})}(x)q_n\subset q_n q'_{n+1} \mathcal{M}
q'_{n+1} q_n = q_n \mathcal{M} q_n.$$
Hence, the restriction $\delta^{(q_n)}|_{q_n\mathcal{M}q_n}$ of the derivation  $\delta^{(q_n)}$ to $q_n\mathcal{M}q_n$ is a derivation on the von Neumann algebra  $q_n\mathcal{M}q_n$. By Sakai-Kadison Theorem
\cite[Theorem 4.1.6]{Sak}, there exists an element $c_n\in
q_n\mathcal{M}q_n$, such that $\delta^{(q_n)}(x)=[c_n,x]$ for all $x\in
q_n\mathcal{M}q_n$.

Now, construct a sequence $\{d_n\}$ of $\mathcal{M}$, such that
\begin{equation}
\label{t1_2_6_1}
\begin{split}
q_nd_mq_n=d_n\ \text{for all}\ n\leq m,\cr
\delta^{(q_n)}(x)=[d_n,x]\ \text{for all}\ x\in q_n\mathcal{M}q_n.
\end{split}
\end{equation}
Set $d_1=c_1$ and assume that elements $d_1,\dots,d_n$ are already
constructed. Since
$\delta^{(q_n)}(q_nxq_n)=q_n\delta^{(q_{n+1})}(q_nxq_n)q_n$, it
follows that
$$[d_n,q_nxq_n]=q_n[c_{n+1},q_nxq_n]q_n=[q_nc_{n+1}q_n,q_nxq_n]$$
for any $x\in\mathcal{M}$. Consequently, the element
$d_n-q_nc_{n+1}q_n$ contained in the centre of algebra
$q_n\mathcal{M}q_n$. By \cite[page 18, corollary]{Dix} there
exists an element $z$ of the centre of algebra
$q_{n+1}\mathcal{M}q_{n+1}$, such that $d_n-q_nc_{n+1}q_n=zq_n$.
Set $d_{n+1}=c_{n+1}+z$. It is clear that
\begin{equation}
\label{t1_2_7}
\delta^{(q_{n+1})}|_{q_{n+1}\mathcal{M}q_{n+1}}(x)=[c_{n+1},x]=[d_{n+1},x]
\end{equation}
for all $x\in q_{n+1}\mathcal{M}q_{n+1}$, in addition,
$$d_{n+1}\in q_{n+1}\mathcal{M}q_{n+1}\ \text{and}\ q_nd_{n+1}q_n=q_nc_{n+1}q_n+zq_n=d_n$$
for every $n\in\mathbb{N}$. Moreover, for $k\in\mathbb{N},\
k<n+1$ the equalities
\begin{equation}
\label{t1_2_8}
q_kd_{n+1}q_k=q_kq_nd_{n+1}q_nq_k=q_kd_nq_k=\dots=q_kd_{k+1}q_k=d_k
\end{equation}
hold.

Thus we constructed the sequence $\{d_n\}$
of elements of $\mathcal{M}$ which has property (\ref{t1_2_6_1}).

By \cite[Prop.8]{CM}, the topology $t(\mathcal{M})$ induces on  $q_nLS(\mathcal{M})q_n=LS(q_n\mathcal{M}q_n)$ the topology $t(q_n\mathcal{M}q_n)$,
and therefore derivation $\delta^{(q_n)}$ is continuous on  $(LS(q_n\mathcal{M}q_n),t(q_n\mathcal{M}q_n))$. By Proposition \ref{p2_2},
we have that $\overline{q_n\mathcal{M}q_n}^{t(q_n\mathcal{M}q_n)}=LS(q_n\mathcal{M}q_n)$. Consequently, the equality $\delta^{(q_n)}(x)=[d_n,x]$ holds for all
 $x\in LS(q_n\mathcal{M}q_n)$.

If $n,m\in\mathbb{N},\ n<m$, then
$$d_m-d_n\stackrel{(\ref{t1_2_6_1})}{=}q_md_mq_m-q_nd_mq_n=(q_m-q_n)d_mq_m+q_nd_m(q_m-q_n).$$
Since
$$r((q_m-q_n)d_mq_m)\sim l((q_m-q_n)d_mq_m)\leq q_n^\bot,$$ it follows that
\begin{equation*}
\begin{split}
 \mathcal{D}(z_n r(d_m-d_n))&\leq \mathcal{D}(z_nr((q_m-q_n)d_mq_m)\vee z_nq_n^\bot)\leq \cr &\leq  2\mathcal{D}(z_nq_n^\bot)\stackrel{(\ref{t1_2_6})}{\leq}
2^{-n+4}\varphi(z_n).
\end{split}
\end{equation*}

From here, in view of (\ref{v2}) and (\ref{t1_2_5}), we obtain
$$d_m-d_n \in V(0,2^{-n},2^{-n+4})\subset V(1/n,2^{-n},2^{-n+4}).$$
It means that $\{d_n\}$ is a Cauchy sequence in
$(LS(\mathcal{M}),t(\mathcal{M}))$, and therefore, since the space
$(LS(\mathcal{M}),t(\mathcal{M}))$ is complete there exists $d\in LS(\mathcal{M})$, such that $d_n\stackrel{t(\mathcal{M})}{\longrightarrow} d$.

Now, let us show that  $\delta(x)=[d,x]$ for all $x\in LS(\mathcal{M})$. By (\ref{t1_2_5}) and (\ref{t1_2_6}) we have that
$q_n^\bot\in V(0,2^{-n},2^{-n+3})$ for all $n\in\mathbb{N}$, and therefore $q_n^\bot\stackrel{t(\mathcal{M})}{\longrightarrow} 0$.
Consequently, $q_n\stackrel{t(\mathcal{M})}{\longrightarrow}\mathbf{1}$ and for every  $x\in LS(\mathcal{M})$ we have that
$q_nxq_n\stackrel{t(\mathcal{M})}{\longrightarrow} x$. We just need to use  $t(\mathcal{M})$-continuity of the derivation $\delta$,
which implies the following
\begin{equation*}
\begin{split}
\delta(x)&= t(\mathcal{M})-\lim_{n\rightarrow\infty} (q_n\delta(q_nxq_n)q_n)=t(\mathcal{M})-\lim_{n\rightarrow\infty}
\delta^{(q_n)}(q_nxq_n) =\cr &= t(\mathcal{M})-\lim_{n\rightarrow\infty} [d_n,q_nxq_n]=[t(\mathcal{M})-\lim_{n\rightarrow\infty} d_n, t(\mathcal{M})-\lim_{n\rightarrow\infty}
q_nxq_n]=\cr &= [d,x].
\end{split}
\end{equation*}
\end{proof}

Theorem \ref{t1} allows to give the full description of all derivations on the algebra $LS(\mathcal{M})$ in case when $\mathcal{M}$ is a properly infinite von Neumann algebra.

\begin{theorem}
\label{t2} If $\mathcal{M}$ is a properly infinite von Neumann algebra, then every derivation on the $*$-algebra $LS(\mathcal{M})$ is inner.
\end{theorem}
\begin{proof}
By \cite[Theorem 3.3]{BCS3} for properly infinite von Neumann algebra $\mathcal{M}$ every derivation
$\delta: LS(\mathcal{M})\rightarrow LS(\mathcal{M})$ is $t(\mathcal{M})$-continuous. Consequently, by Theorem \ref{t1},
there exists  $d\in LS(\mathcal{M})$, such that  $\delta(x)=[d,x]$ for all $x\in LS(\mathcal{M})$.
\end{proof}

\section{Derivations on $EW^*$-algebras}
In this section we give applications of Theorems \ref{t1} and \ref{t2} to the description of continuous derivations on $EW^*$-algebras. The class of $EW^*$-algebras (extended $W^*$-algebras) was introduced in \cite{Dixon1} for the purpose of description of $*$-algebras of unbounded closed operators, which are "similar" to $W^*$-algebras by their algebraic and order properties.

Let $\mathcal{A}$ be a set of closed, densely defined operators on
the Hilbert space $H$ which is a $*$-algebra under strong sum,
strong product, scalar multiplication and the usual adjoint of
operators. The set $\mathcal{A}$ is said to be $EW^*$-algebra
\cite{Dixon1} if the following conditions hold:

(i) $(\mathbf{1}+x^*x)^{-1}\in\mathcal{A}$ for every
$x\in\mathcal{A}$;

(ii) the subalgebra $\mathcal{A}_b$ of bounded operators in $\mathcal{A}$ is a $W^*$-algebra.

In \cite{CZ} it is given the meaningful connection between $EW^*$-algebras $\mathcal{A}$ and solid subalgebras of $LS(\mathcal{A}_b)$.
Recall \cite{BdPS}, that a $*$-subalgebra $\mathcal{A}$ of $LS(\mathcal{M})$ is called solid if conditions
$x\in LS(\mathcal{M}),\ y\in\mathcal{A},\ |x|\leq |y|$ imply that $x\in\mathcal{A}$. It is clear that every solid $*$-subalgebra $\mathcal{A}$ in $LS(\mathcal{M})$ with $\mathcal{M}\subset\mathcal{A}$ is an $EW^*$-algebra and $\mathcal{A}_b=\mathcal{M}$. At the same time, in \cite{CZ} it is established that every $EW^*$-algebra $\mathcal{A}$ with the bounded part $\mathcal{A}_b=\mathcal{M}$ is a solid
$*$-subalgebra in the $*$-algebra $LS(\mathcal{M})$, i.e.
$LS(\mathcal{M})$ is the greatest $EW^*$-algebra of all
$EW^*$-algebras with the bounded part coinciding with
$\mathcal{M}$.

Since every $EW^*$-algebra $\mathcal{A}$ with the bounded part $\mathcal{A}_b$ is a solid $*$-subalgebra in $LS(\mathcal{A}_b)$ and $\mathcal{A}_b\subset \mathcal{A}$, according to \cite{BCS3}, it follows that in the case when $\mathcal{A}_b$ is a properly infinite $W^*$-algebra any derivation $\delta\colon\mathcal{A}\to LS(\mathcal{A}_b)$ is continuous with respect to the local measure topology $t(\mathcal{A}_b)$.

Now, let $\mathcal{A}_b$ be arbitrary $W^*$-algebra and let
$\delta\colon\mathcal{A}\to\mathcal{A}$ be a
$t(\mathcal{A}_b)$-continuous derivation. Since
$\mathcal{A}_b\subset\mathcal{A}$, $\mathcal{A}_b$ is everywhere
dense in \\ $(LS(\mathcal{A}_b),t(\mathcal{A}_b))$ (Proposition
\ref{p2_2}) and $(LS(\mathcal{A}_b),t(\mathcal{A}_b))$ is a
topological $*$-algebra, there exists a unique
$t(\mathcal{A}_b)$-continuous derivation $\hat{\delta}\colon
LS(\mathcal{A}_b)\to LS(\mathcal{A}_b)$ such that
$\hat{\delta}(x)=\delta(x)$ for all $x\in\mathcal{A}$. By Theorem
\ref{t1} the derivation $\hat{\delta}$ is inner. In
\cite[Proposition 5.13]{BdPS} it is proved that, if $\delta$ is a
derivation on a solid $*$-subalgebra
$\mathcal{A}\supset\mathcal{M}$ and $\delta(x)=[w,x]$ for all
$x\in\mathcal{A}$ and some $w\in LS(\mathcal{M})$, then there
exists $w_1\in\mathcal{A}$, such that $\delta(x)=[w_1,x]$ for all
$x\in\mathcal{A}$, i.e. the derivation $\delta$ is inner on the
$*$-subalgebra $\mathcal{A}$.

Thus the following theorem holds.

\begin{theorem}\label{t_EW}
(i) Every $t(\mathcal{A}_b)$-continuous derivation on a $EW^*$-algebra $\mathcal{A}$ is inner;

(ii) If the bounded part $\mathcal{A}_b$ in an $EW^*$-algebra
$\mathcal{A}$ is a properly infinite $W^*$-algebra, then every
derivation on $\mathcal{A}$ is inner.
\end{theorem}

%

Let $\mathcal{M}$ be a semifinite von Neumann algebra acting on the
Hilbert space $H$, $\tau$ be a faithful normal semifinite trace on
$\mathcal{M}$. An operator $x\in S(\mathcal{M})$ with the domain
$\mathfrak{D}(x)$ is called \emph{$\tau$-measurable} if for any
$\varepsilon>0$ there exists a projection
$p\in\mathcal{P}(\mathcal{M})$ such that $p(H)\subset
\mathfrak{D}(x)$ and $\tau(p^\bot)<\varepsilon$.

The set $S(\mathcal{M},\tau)$ of all $\tau$-measurable operators
is a solid  $*$-subalgebra of $LS(\mathcal{M})$ such
that $\mathcal{M}\subset S(\mathcal{M},\tau)\subset
S(\mathcal{M})$. If the trace $\tau$ is finite, then
$S(\mathcal{M},\tau)=S(\mathcal{M})$. The algebra
$S(\mathcal{M},\tau)$ is a noncommutative version of the algebra
of all measurable complex functions $f$ defined on
$(\Omega,\Sigma,\mu)$, for which $\mu(\{|f|>\lambda\})\rightarrow
0$ as $\lambda\rightarrow\infty$.

Let $t_\tau$ be the \textit{measure topology}~\cite{Ne} on
$S(\mathcal{M},\tau)$ whose base of neighborhoods of zero is given
by
\begin{gather*}
  U(\varepsilon,\delta)=\{x\in S(\mathcal{M},\tau): \ \
  \hbox{there exists a projection} \ \
  p\in \mathcal{P}(\mathcal{M}),
\\
  \hbox{such that}\ \tau(p^\bot)\leq \delta,
  \ xp \in \mathcal{M}, \ \ \|xp\|_\mathcal{M}\leq\varepsilon\}, \ \ \varepsilon>0, \ \delta>0.
\end{gather*}

The pair $(S(\mathcal{M},\tau),t_\tau)$ is a complete metrizable
topological $*$-algebra. Here, the topology $t_\tau$ majorizes the
topology $t(\mathcal{M})$ on $S(\mathcal{M},\tau)$ and, if $\tau$
is a finite trace, the topologies $t_\tau$ and $t(\mathcal{M})$
coincide~\cite[\S\S\,3.4, 3.5]{MCh}. Denote by
$t(\mathcal{M},\tau)$ the topology on $S(\mathcal{M},\tau)$
induced by the topology $t(\mathcal{M})$. It is not true in
general that, if the topologies $t_\tau$ and $t(\mathcal{M},\tau)$
are the same, then the von Neumann algebra $\mathcal{M}$ is
finite. Indeed, if $\mathcal{M}=B(H), \ \dim(H)=\infty$, $\tau=tr$
is the canonical trace on $B(H)$, then
$LS(\mathcal{M})=S(\mathcal{M})=S(\mathcal{M},\tau)=\mathcal{M}$,
and the two topologies $t_\tau$ and $t(\mathcal{M})$ coincide with
the uniform topology on $B(H)$.

At the same time, if $\mathcal{M}$ is a finite von Neumann algebra
with a faithful normal semifinite trace $\tau$ and $t_\tau =
t(\mathcal{M},  \tau)$, then $\tau(I)<\infty$ \cite{CM}.

In \cite{Ber2} it is proven that every $t_\tau$-continuous derivation on $S(\mathcal{M},\tau)$ is inner. In addition, in \cite{Ber} it is established that in the case of properly infinite von Neumann
algebra $\mathcal{M}$ every derivation on $S(\mathcal{M},\tau)$ is
$t_\tau$-continuous. Thus, in view if Theorem \ref{t_EW} we obtain the following

\begin{corollary}
\label{cor_smt} Let $\mathcal{M}$ be a semifinite von Neumann
algebra, let $\tau$ be a faithful normal semifinite trace on
$\mathcal{M}$, let $\delta$ be a derivation on
$S(\mathcal{M},\tau)$. Then the following conditions are equivalent:

(i). $\delta$ is $t(\mathcal{M})$-continuous;

(ii). $\delta$ is $t_\tau$-continuous;

(iii). $\delta$ is inner.

In addition, if $\mathcal{M}$ is a properly infinite von Neumann
algebra then every derivation on $S(\mathcal{M},\tau)$ is inner.
\end{corollary}

\section{Automatic innerness of derivations on Banach $\mathcal{M}$-bimodule of locally measurable operators}

In this section we give one more application of Theorem \ref{t1} establishing innerness of every derivation on a Banach
$\mathcal{M}$-bimodule of locally measurable operators.

Let $\mathcal{M}$ be a von Neumann algebra. A linear subspace
$\mathcal{E}$ of $LS(\mathcal{M})$, is called a
$\mathcal{M}$-bimodule of locally measurable operators if $uxv\in
\mathcal{E}$ whenever $x\in \mathcal{E}$ and $u,v\in \mathcal{M}$.
If $\mathcal{E}$ is a $\mathcal{M}$-bimodule of locally measurable
operators, $x\in\mathcal{E}$ and $x=v|x|$ is the polar
decomposition of operator $x$ then
$|x|=v^*x\in\mathcal{\mathcal{E}}$ and $x^*=|x|v^*\in\mathcal{E}$.
In addition,
\begin{equation}
\label{e_solid}\text{if}\ |a|\leq |b|,\ b\in\mathcal{E},\ a\in
LS(\mathcal{M})\ \text{then}\ a\in\mathcal{E}.
\end{equation}
Property (\ref{e_solid}) of a $\mathcal{M}$-bimodule of locally
measurable operators follows from the following proposition.

\begin{proposition}
\label{mchext} Let $\mathcal{M}$ be a von Neumann algebra acting in a Hilbert space $H$, $a,b\in LS(\mathcal{M}),\ 0\leq a\leq b$. Then $a^{1/2}=cb^{1/2}$ for some $c\in s(b)\mathcal{M}s(b),\ \|c\|_{\mathcal{M}}\leq 1$, in particular, $a=cbc^*$. In addition, if $c_1\in\mathcal{M}$ and \
$a^{1/2}=c_1b^{1/2}$, then $s(b)\cdot c_1\cdot s(b)=c$.
\end{proposition}
\begin{proof}
Let us show firstly that $s(a)\leq s(b)$. Since
$$0\leq(\mathbf{1}-s(b))a(\mathbf{1}-s(b))\leq
(\mathbf{1}-s(b))b(\mathbf{1}-s(b))=0,$$ if follows that
$(\mathbf{1}-s(b))a^{1/2}=0$, that implies the equality
$(\mathbf{1}-s(b))a=0$, i.e. $s(b)a=a=a^*=a^*s(b)=as(b)$.
Consequently, $s(a)\leq s(b)$.

Thus, passing if necessary to the reduction $s(b)\mathcal{M}s(b)$
we may assume that $s(b)=\mathbf{1}$.

For every $n\in\mathbb{N}$ denote by $p_n$ the spectral projection for the operator $b$ corresponding to the interval $[1/n,n]$. Since
$p_n\uparrow s(b)=\mathbf{1}$ it follows that the linear subspace
$H_0=\bigcup_{n=1}^\infty p_n H$ is dense in $H$ and $H_0\subset
\mathfrak{D}(b)\cap\mathfrak{D}(b^{1/2})$. Furthermore, according to the inequalities $0\leq p_nap_n\leq p_nbp_n\leq np_n$ we have that
$a^{1/2}p_n\in\mathcal{M}$ and $\|a^{1/2}p_n\|_{\mathcal{M}}\leq
\sqrt{n}$ for all $n\in\mathbb{N}$. In particular, $H_0\subset
\mathfrak{D}(a^{1/2})$.

Since $b^{1/2}p_n\leq n^{1/2}p_n$ and
$b^{1/2}(p_nH)=p_nb^{1/2}(p_nH)\subset p_nH$ for all
$n\in\mathbb{N}$ we have $b^{1/2}(H_0)\subset H_0$. Consequently, it is possible to define a linear mapping $d:b^{1/2}(H_0)\rightarrow H$ by setting
$d(b^{1/2}\xi)=a^{1/2}\xi,\ \xi\in H_0$. The definition of the operator $d$ is correct since the equality $b^{1/2}\xi=0$ and inequality
$$\|a^{1/2}\xi\|_{H}^2=(a^{1/2}\xi,a^{1/2}\xi)=(a\xi,\xi)\leq
(b\xi,\xi)=\|b^{1/2}\xi\|_{H}^2$$ imply that $a^{1/2}\xi=0$.

In addition, for every $\xi\in H_0$ we have
$$\|d(b^{1/2}\xi)\|_{H}^2=\|a^{1/2}\xi\|_{H}^2\leq
\|b^{1/2}\xi\|_{H}^2,$$ i.e. $d$ is a continuous linear operator on $b^{1/2}(H_0)$ and  $\|d\|_{b^{1/2}(H_0)\rightarrow
H}\leq 1$.

Since $n^{-1}p_n\leq bp_n\leq np_n$, by Proposition \cite[Теорема 2.4.2]{MCh} we have $n^{-1/2}p_n\leq b^{1/2}p_n\leq n^{1/2}p_n$. Therefore the restriction of operator $b^{1/2}$ to $p_n(H_0)$ has inverse bounded operator $b_n$, in addition $n^{-1/2}p_n\leq b_np_n\leq n^{1/2}p_n$. Hence, $b^{1/2}(p_nH)=p_nH$, that implies the equality $b^{1/2}(H_0)=H_0$.

Thus, the operator $d$ uniquely extends to the Hilbert space $H$ up to a bounded linear operator $c$, moreover, $\|c\|_{B(H)}\leq 1$ and $cb^{1/2}\xi=a^{1/2}\xi$ for all $\xi\in H_0$.

If $u$ as a unitary operator from the commutant $\mathcal{M}'$, then $u(p_nH)=p_nH$ for all $n\in\mathbb{N}$ and therefore $u(H_0)=H_0$. If $\eta\in H_0$, then $\eta=b^{1/2}\xi$ for some $\xi\in H_0$ and
\begin{gather*}
\begin{split}
u^{-1}cu\eta&=u^{-1}cub^{1/2}\xi=u^{-1}cb^{1/2}u\xi\\ &=u^{-1}a^{1/2}u\xi=u^{-1}ua^{1/2}\xi=a^{1/2}\xi=cb^{1/2}\xi=c\eta.
\end{split}
\end{gather*}
Consequently, $u^{-1}cu=c$, that implies the inclusion $c\in\mathcal{M}$.

Since $p_ncb^{1/2}p_n=p_na^{1/2}p_n$ for all $n\in\mathbb{N}$ and $p_n\uparrow\mathbf{1}$, by Proposition \cite[Proposition 2.4.1 (ix)]{MCh} we have
 $cb^{1/2}=a^{1/2}$.

If $c_1\in\mathcal{M}$ and $c_1b^{1/2}=a^{1/2}$, then the operators $c_1$ and $c$ coincide on the everywhere dense subspace $H_0$ and therefore $c_1=c$.

If $s(b)\neq \mathbf{1},\ c_1\in\mathcal{M}$ and $c_1b^{1/2}=a^{1/2}$, then, using inequalities $$a^{1/2}s(b)=s(b)a^{1/2}=a^{1/2}$$ and $$b^{1/2}s(b)=s(b)b^{1/2}=b^{1/2},$$ we obtain $(s(b)c_1s(b))b^{1/2}=a^{1/2}$. Uniqueness of the operator $c$ in reduction $s(b)\mathcal{M}s(b)$ implies that $s(b)\cdot c_1\cdot s(b)=c$.
\end{proof}

Let $\mathcal{E}$ be a $\mathcal{M}$-bimodule of locally measurable operators.
A linear mapping  $\delta:\mathcal{M}\rightarrow \mathcal{E}$ is called \emph{derivation}, if
$\delta(ab)=\delta(a)b+a\delta(b)$ for all $a,b\in\mathcal{M}$. A derivation $\delta:\mathcal{M}\longrightarrow \mathcal{E}$ is called
\emph{inner}, if there exists an element $d\in\mathcal{E}$, such that $\delta(x)=[d,x]=dx-xd$ for all $x\in\mathcal{M}$.

We need the following

\begin{theorem} \cite[Theorem 1]{BS}
\label{t_bs} Let $\mathcal{M}$ be a von Neumann algebra and $a\in
LS_h(\mathcal{M})$. Then there exist a self-adjoint operator
$c$ in the centre of the $*$-algebra $LS(\mathcal{M})$ and a family
$\{u_\varepsilon\}_{\varepsilon>0}$ of unitary operators from
$\mathcal{M}$ such that
\begin{equation}
\label{e_bs}
|[a,u_\varepsilon]|\geq (1-\varepsilon)|a-c|.
\end{equation}
\end{theorem}

The following theorem establishes innerness of every derivation $\delta: \mathcal{M}\rightarrow\mathcal{E}$ in case of properly infinite von Neumann algebra
$\mathcal{M}$.

\begin{theorem}
\label{t4} Let $\mathcal{M}$ be a properly infinite von Neumann algebra
and let $\mathcal{E}$ be a $\mathcal{M}$-bimodule of locally measurable operators.
Then any derivation $\delta:\mathcal{M}\longrightarrow \mathcal{E}$ is inner.
\end{theorem}
\begin{proof}
By \cite[Theorem 4.8]{BCS3} there exists a derivation
$\overline{\delta}: LS(\mathcal{M})\rightarrow
LS(\mathcal{M})$, such that $\overline{\delta}(x)=\delta(x)$
for all $x\in\mathcal{M}$. By Theorem \ref{t2}, there exists an element  $a\in LS(\mathcal{M})$, such that
$\overline{\delta}(x)=[a,x]$ for all $x\in LS(\mathcal{M})$.
It is clear that
$[a,\mathcal{M}]=\overline{\delta}(\mathcal{M})=\delta(\mathcal{M})\subset
\mathcal{E}$.

Let $a_1=Re(a),\ a_2=Im(a)$. Since $[a^*,x]=-[a,x^*]^*\in\mathcal{E}$ for any $x\in\mathcal{M}$, it follows that $[a_1,x]=[a+a^*,x]/2\in\mathcal{E}$ and
$[a_2,x]=[a-a^*,x]/2i\in\mathcal{E}$ for all $x\in\mathcal{M}$.

By Theorem \ref{t_bs} and by taking $\varepsilon=1/2$ in (\ref{e_bs}) we obtain that there exist
$c_1,c_2\in\mathcal{Z}_h(LS(\mathcal{M}))$ and unitary operators
$u_1,u_2\in\mathcal{M}$ such that
$$2|[a_i,u_i]|\geq |a_i-c_i|,\ i=1,2.$$ Since
$[a_i,u_i]\in\mathcal{E}$ and $\mathcal{E}$ is
$\mathcal{M}$-bimodule we have that $d_i:=a_i-c_i\in\mathcal{E}$, $i=1,2$
(see (\ref{e_solid})). Therefore $d=d_1+id_2\in\mathcal{E}$.
Since $c_1,c_2$ are central projections from $LS(\mathcal{M})$ it follows that $\delta(x)=[a,x]=[d,x]$ for all $x\in\mathcal{M}$.
\end{proof}

Let $\mathcal{M}$ be a von Neumann algebra. If a $\mathcal{M}$-bimodule of locally measurable operators $\mathcal{E}$ is equipped with a norm
$\left\Vert \cdot \right\Vert _{\mathcal{E}}$, satisfying
\begin{equation}
\left\Vert uxv\right\Vert _{\mathcal{E}}\leq \left\Vert u\right\Vert
_{\mathcal{M}}\left\Vert v\right\Vert _{\mathcal{M}}\left\Vert
x\right\Vert _{\mathcal{E}},\ \ \ x\in \mathcal{E},\ u,v\in \mathcal{M}\text{,}
\label{ChIVeq21}
\end{equation}%
then $\mathcal{E}$ is called a \emph{ normed $\mathcal{M}$-bimodule of locally measurable operators}.
If, in addition, $(\mathcal{E},\|\cdot\|_{\mathcal{E}})$ is a Banach space, then $\mathcal{E}$ is called a \emph{ Banach $\mathcal{M}$-bimodule of locally measurable operators}.

Easy to see that for the norm $\|\cdot\|_{\mathcal{E}}$ on a normed $\mathcal{M}$-bimodule of locally measurable operators $\mathcal{E}$ the following properties hold:
\begin{equation}
\label{e1}
\||a|\|_\mathcal{E}=\|a^*\|_\mathcal{E}=\|a\|_\mathcal{E}\ \text{for any}\ a\in\mathcal{E};
\end{equation}
\begin{equation}
\label{e2}
\|a\|_\mathcal{E}\leq\|b\|_\mathcal{E}\ \text{for any}\ a,b\in\mathcal{E},\ 0\leq a\leq b;
\end{equation}
\begin{equation}
\label{e3}
\begin{split}
\text{If}\ q\in\mathcal{E}\cap\mathcal{P}(\mathcal{M}),\ p\in \mathcal{P}(\mathcal{M}),\ p\preceq q,\\ \text{ then }\
p\in\mathcal{E}\ \text{and}\ \|p\|_\mathcal{E}\leq\|q\|_\mathcal{E}.
\end{split}
\end{equation}

\begin{proposition}
\label{pe4}
If $\{p_k\}_{k=1}^n\subset \mathcal{P}(\mathcal{M})\cap\mathcal{E}$ then
\begin{equation}
\label{e4}
\bigvee_{k=1}^n p_n\in\mathcal{E}\ \text{and}\ \|\bigvee_{k=1}^n p_n\|_\mathcal{E}\leq\sum_{k=1}^n\|p_k\|_\mathcal{E}.
\end{equation}
\end{proposition}
\begin{proof}
If $p,q\in\mathcal{P}(\mathcal{M})\cap\mathcal{E}$, then $p\vee
q-p\sim q-p\wedge q\leq q$ and therefore, $p\vee q-p\in\mathcal{E}$
and $\|p\vee q-p\|_{\mathcal{E}}\leq\|q\|_{\mathcal{E}}$ (see
(\ref{e3})). Hence, $p\vee q=(p\vee q-p)+p\in\mathcal{E}$
and $\|p\vee q\|_{\mathcal{E}}-\|p\|_{\mathcal{E}}\leq\|p\vee
q-p\|_{\mathcal{E}}\leq\|q\|_{\mathcal{E}}$.

For an arbitrary finite set $\{p_k\}_{k=1}^n\subset \mathcal{P}(\mathcal{M})\cap\mathcal{E}$ proposition (\ref{e4}) is proved using  mathematical induction.
\end{proof}

In Lemmas \ref{l_e1}-\ref{l_e5} given below we assume that on a von Neumann algebra $\mathcal{M}$ there is a faithful normal finite trace $\tau$. In this case, the algebra $\mathcal{M}$ is finite. Moreover,
$LS(\mathcal{M})=S(\mathcal{M})=S(\mathcal{M},\tau)$,
$t(\mathcal{M})=t_\tau$ and $(LS(\mathcal{M}),t(\mathcal{M}))$ is
an $F$-space.

Later we need the following

\begin{proposition} \cite[Prop.3.5.7(i)]{MCh}
\label{p_mchtau} Let $\mathcal{M}$ be a von Neumann algebra with faithful normal finite trace $\tau$ and $\{p_n\}_{n=1}^\infty\subset\mathcal{P}(\mathcal{M})$. Then
$$p_n\stackrel{t(\mathcal{M})}{\longrightarrow}0\Leftrightarrow\tau(p_n)\rightarrow 0.$$
\end{proposition}

Let $\mathcal{E}$ be a Banach $\mathcal{M}$-bimodule of locally
measurable operators in $LS(\mathcal{M})$.

\begin{lemma}
\label{l_e1} If $\{p_n\}_{n=1}^\infty\subset\mathcal{P}(\mathcal{M})\cap\mathcal{E}$ and the series $\sum_{n=1}^\infty\|p_n\|_{\mathcal{E}}$ converges, then
$p=\bigvee_{n=1}^\infty p_n\in\mathcal{E}$ and $\|p\|_{\mathcal{E}}\leq \sum_{n=1}^\infty\|p_n\|_{\mathcal{E}}$.
\end{lemma}
\begin{proof}
Set $q_n=\bigvee_{k=1}^n p_k$. According to (\ref{e4}), $q_n\in\mathcal{E}$ and $\|q_n\|_{\mathcal{E}}\leq\sum_{k=1}^n\|p_k\|_{\mathcal{E}}$.

Let $n,m\in\mathbb{N},\ n<m$. By (\ref{e3}) and (\ref{e4}) we have that
\begin{gather*}
\begin{split}
\|q_m-q_n\|_{\mathcal{E}}&=\|q_n\vee\bigvee_{k=n+1}^m p_k-q_n\|_{\mathcal{E}}\\
&=\|\bigvee_{k=n+1}^m p_k-q_n\wedge\bigvee_{k=n+1}^m p_k\|_{\mathcal{E}}\leq\|\bigvee_{k=n+1}^m p_k\|_{\mathcal{E}}\leq
\sum_{k=n+1}^m\|p_k\|_{\mathcal{E}}.
\end{split}
\end{gather*}
Consequently, $\{q_n\}$ is a Cauchy sequence in  $(\mathcal{E},\|\cdot\|_{\mathcal{E}})$, and therefore there exists $q\in\mathcal{E}$, such that $\|q_n-q\|_{\mathcal{E}}\rightarrow 0$, in addition
$\|q\|_{\mathcal{E}}\leq\sum_{n=1}^\infty\|p_n\|_{\mathcal{E}}$.

Since
$$\|qp-q_n\|_{\mathcal{E}}=\|qp-q_np\|_{\mathcal{E}}\leq\|p\|_{\mathcal{M}}\|q-q_n\|_{\mathcal{E}},$$ it follows that $qp=q=q^*=pq$. Hence,
$s(p-q)\leq p$. Fix $n_0\in\mathbb{N}$, then for $n>n_0$, we have
\begin{gather*}
\begin{split}
\|q_{n_0}q-q_{n_0}\|_{\mathcal{E}}&=\|q_{n_0}q-q_{n_0}q_n\|_{\mathcal{E}}\\
&\leq
\|q_{n_0}\|_{\mathcal{M}}\|q-q_n\|_{\mathcal{E}}\leq\|q-q_n\|_{\mathcal{E}}.
\end{split}
\end{gather*}
Passing to the limit for $n\rightarrow\infty$, we obtain
$q_{n_0}q=q_{n_0}$. Therefore, $q_n(p-q)q_n=0$ for all
$n\in\mathbb{N}$. The inequality $s(p-q)\leq p$ and convergence
$q_n\uparrow p$ by \cite[Prop. 2.4.1(ix)]{MCh} imply that $q=p$.
\end{proof}

\begin{lemma}
\label{l_e2} If
$\{a_n\}_{n=1}^\infty\subset \mathcal{E}$ and $\|a_n\|_{\mathcal{E}}\rightarrow 0$,
then $a_n\stackrel{t(\mathcal{M})}{\longrightarrow}0$.
\end{lemma}
\begin{proof}
It is sufficient to show that every convergent to zero in the norm $\|\cdot\|_{\mathcal{E}}$ sequence from $\mathcal{E}$ has a subsequence convergent to zero in the topology $t(\mathcal{M})$.

Firstly, consider a sequence
$\{p_n\}_{n=1}^\infty\in\mathcal{P}(\mathcal{M})\cap\mathcal{E}$,
such that $\|p_n\|_{\mathcal{E}}\rightarrow 0$. Choose a subsequence $\{p_{n_k}\}_{k=1}^\infty$ so that
$\|p_{n_k}\|_{\mathcal{E}}\leq 2^{-k}$. By Lemma \ref{l_e1} for the sequence of projections $q_k=\sup_{l\geq k+1}p_{n_l}$we have $q_k\in\mathcal{E}$ and $\|q_k\|_{\mathcal{E}}\leq
2^{-k}$. If $q=\inf_{k\geq 1}q_k$, then $q\in\mathcal{E}$ and
$\|q\|_{\mathcal{E}}\leq\|q_k\|_{\mathcal{E}}\leq 2^{-k}$ for all
$k\in\mathbb{N}$, that implies $q=0$. Consequently,
$q_k\downarrow 0$, and therefore $\tau(q_k)\downarrow 0$.

Since $p_{n_{k+1}}\leq q_k$ for all $k\in\mathbb{N}$ we have
$\tau(p_{n_k})\rightarrow 0$, that by Proposition
\ref{p_mchtau} implies the convergence
$p_{n_k}\stackrel{t(\mathcal{M})}{\longrightarrow}0$. Thus, every sequence
$\{p_n\}_{n=1}^\infty\in\mathcal{P}(\mathcal{M})\cap\mathcal{E}$ convergent to zero in the norm $\|\cdot\|_{\mathcal{E}}$ automatic converges to zero in the topology $t(\mathcal{M})$.

Now, let $\{a_n\}_{n=1}^\infty\subset \mathcal{E}$ and $\|a_n\|_{\mathcal{E}}\rightarrow 0$. For every $\lambda>0$ inequality
$\lambda E_\lambda^\bot(|a_n|)\leq |a_n|E_\lambda^\bot(|a_n|)\leq |a_n|$ imply that
$$\|E_\lambda^\bot(|a_n|)\|_{\mathcal{E}}\stackrel{(\ref{e2})}{\leq}\lambda^{-1} \||a_n|\|_{\mathcal{E}}\stackrel{(\ref{e1})}{\leq}\lambda^{-1}\|a_n\|_{\mathcal{E}}\rightarrow 0.$$
According to the proven above we have that $E_\lambda^\bot(|a_n|)\stackrel{t(\mathcal{M})}{\longrightarrow}0$. Finally, by  Proposition \ref{p2_1} (ii) we obtain  $a_n\stackrel{t(\mathcal{M})}{\longrightarrow}0$.
\end{proof}

\begin{lemma}
\label{l_e4} If $\{a_n\}_{n=1}^\infty\subset LS(\mathcal{M})$ and $a_n\stackrel{t(\mathcal{M})}{\longrightarrow}0$, then there exists a sequence $\{a_{n_k}\}_{k=1}^\infty$ such that $a_{n_k}=b_k+c_k$, where $b_k\in\mathcal{M},\ c_k\in LS(\mathcal{M}),\ k\in\mathbb{N},\ \|b_k\|_{\mathcal{M}}\rightarrow 0$ and $s(|c_k|)\stackrel{t(\mathcal{M})}{\longrightarrow}0$.
\end{lemma}
\begin{proof}
Since $(LS(\mathcal{M}),t(\mathcal{M}))$ is an $F$-space there exists a countable basis $\{U_k\}_{k=1}^\infty$ of neighborhoods of zero of the topology $t(\mathcal{M})$.

By Proposition \ref{p2_1} (ii) we have
$E_\lambda^\bot(|a_n|)\stackrel{t(\mathcal{M})}{\longrightarrow}0$
for every $\lambda>0$. Therefore, there exists a sequence  $a_{n_k}$ such that  $E_{1/k}^\bot(|a_{n_k}|)\in
U_k$ for all $k\in\mathbb{N}$. Set
$b_k=a_{n_k}E_{1/k}(|a_{n_k}|)$ and
$c_k=a_{n_k}E_{1/k}^\bot(|a_{n_k}|)$. It is clear that
$b_k\in\mathcal{M}$ and $\|b_k\|_{\mathcal{M}}\leq 1/k$. Since
\begin{gather*}
\begin{split}
|c_k|&=(c_k^*c_k)^{1/2}= (E_{1/k}^\bot(|a_{n_k}|)|a_{n_k}|^2
E_{1/k}^\bot(|a_{n_k}|))^{1/2}
\\&= E_{1/k}^\bot(|a_{n_k}|)|a_{n_k}|
E_{1/k}^\bot(|a_{n_k}|) =|a_{n_k}|E_{1/k}^\bot(|a_{n_k}|),
\end{split}
\end{gather*}
it follows that
$$s(|c_k|)\leq E_{1/k}^\bot(|a_{n_k}|)\in U_k.$$
Since
$\{U_k\}_{k=1}^\infty$ is a basis of neighborhoods of zero of the topology
$t(\mathcal{M})$ we have
$E_{1/k}^\bot(|a_{n_k}|)\stackrel{t(\mathcal{M})}{\longrightarrow}0$,
that implies the convergence
$\tau(E_{1/k}^\bot(|a_{n_k}|))\rightarrow 0$ (Proposition
\ref{p_mchtau}). From the inequality
$\tau(s(|c_k|))\leq\tau(E_{1/k}^\bot(|a_{n_k}|))$ and Proposition
\ref{p_mchtau} we obtain
$s(|c_k|)\stackrel{t(\mathcal{M})}{\longrightarrow}0$.
\end{proof}

\begin{lemma}
\label{l_e5} Every derivation $\delta: LS(\mathcal{M})\rightarrow LS(\mathcal{M})$ with $\delta(\mathcal{M})\subset\mathcal{E}$ is $t(\mathcal{M})$-continuous.
\end{lemma}
\begin{proof}
Since $(LS(\mathcal{M}),t(\mathcal{M}))$ is an $F$-space it is sufficient to show that the graph of the linear operator $\delta$ is closed.

Suppose that the graph of the operator $\delta$ is not closed. Then there exists a sequence $\{a_n\}_{n=1}^\infty\subset LS(\mathcal{M})$ and $0\neq b\in LS(\mathcal{M})$ such that
$a_n\stackrel{t(\mathcal{M})}{\longrightarrow}0$ and $\delta(a_n)\stackrel{t(\mathcal{M})}{\longrightarrow}b$.

According to Lemma \ref{l_e4} and passing, if necessary, to a subsequence, we may assume that $a_n=b_n+c_n$, where $b_n\in\mathcal{M},\ c_n\in LS(\mathcal{M}),\ n\in\mathbb{N},\ \|b_n\|_{\mathcal{M}}\rightarrow 0$ and
$s(|c_n|)\stackrel{t(\mathcal{M})}{\longrightarrow}0$ for $n\rightarrow\infty$.

Since the restriction $\delta|_{\mathcal{M}}$ of the derivation $\delta$ to the von Neumann algebra $\mathcal{M}$ is a derivation from $\mathcal{M}$ into the Banach $\mathcal{M}$-bimodule, by Ringrose Theorem \cite{Ringrose} we have
$\|\delta(b_n)\|_{\mathcal{E}}\rightarrow 0$. Lemma \ref{l_e2} implies that $\delta(b_n)\stackrel{t(\mathcal{M})}{\longrightarrow}0$.

From the inequalities
$$\delta(c_n)=\delta(c_ns(|c_n|))=\delta(c_n)s(|c_n|)+c_n\delta(s(|c_n|))$$
we have that
$$s(\delta(c_n))\leq l(\delta(c_n)s(|c_n|))\vee r(\delta(c_n)s(|c_n|))\vee l(c_n\delta(s(|c_n|))) \vee r(c_n\delta(s(|c_n|))).$$
Since
$$l(c_n)\sim r(c_n)=s(|c_n|),\ l(\delta(c_n)s(|c_n|))\sim r(\delta(c_n)s(|c_n|))\leq s(|c_n|),$$
$$r(c_n\delta(s(|c_n|)))\sim l(c_n\delta(s(|c_n|)))\leq l(c_n)\preceq s(|c_n|),$$
it follows that
$$\tau(s(\delta(c_n)))\leq 4\tau(s(|c_n|)).$$
By Proposition \ref{p_mchtau},
$\tau(s(|c_n|))\rightarrow 0$, and therefore
$\tau(s(\delta(c_n)))\rightarrow 0$ and
$\tau(s(|\delta(c_n)|))\rightarrow 0$, that implies the convergence
$\tau(E_\lambda^\bot(|\delta(c_n)|))\rightarrow 0$ for every
$\lambda>0$. Hence by Propositions \ref{p2_1} (ii) and
\ref{p_mchtau}, we obtain
$\delta(c_n)\stackrel{t(\mathcal{M})}{\longrightarrow}0$.

Thus,
$\delta(a_n)=\delta(b_n)+\delta(c_n)\stackrel{t(\mathcal{M})}{\longrightarrow}0$, that contradicts to the inequality  $b\neq 0$. Consequently, the operator
$\delta$ has a closed graph, therefore $\delta$ is
$t(\mathcal{M})$-continuous.
\end{proof}

Now, we give the main result of this section.

\begin{theorem}
\label{t_mbm} Let $\mathcal{M}$ be a von Neumann algebra and let
$\mathcal{E}$ be a Banach  $\mathcal{M}$-bimodule of local
measurable operators. Then any derivation
$\delta:\mathcal{M}\rightarrow \mathcal{E}$ is inner. In
addition, there exist $d\in\mathcal{E}$ such that
$\delta(x)=[d,x]$ for all $x\in\mathcal{M}$ and
$\|d\|_{\mathcal{E}}\leq
2\|\delta\|_{\mathcal{M}\rightarrow\mathcal{E}}$. If
$\delta^*=\delta$ or $\delta^*=-\delta$ then $d$ may be chosen so that
$\|d\|_{\mathcal{E}}\leq\|\delta\|_{\mathcal{M}\rightarrow\mathcal{E}}$.
\end{theorem}
\begin{proof}
According to \cite[Theorem 4.8]{BCS3} there exists a derivation
$\overline{\delta}: LS(\mathcal{M})\rightarrow
LS(\mathcal{M})$ such that $\overline{\delta}(x)=\delta(x)$
for all $x\in\mathcal{M}$.

Choose a central decomposition of unit $\{z_\infty,z_i\}_{j\in J}$ such that $\mathcal{M}z_\infty$ is a properly infinite von Neumann algebra and on every von Neumann algebra $\mathcal{M}z_j$ there exists a faithful normal finite trace.
By \cite[Theorem 3.3]{BCS3} the derivation  $\overline{\delta}^{(z_\infty)}:=\overline{\delta}|_{LS(\mathcal{M}z_\infty)}:LS(\mathcal{M}z_\infty)\rightarrow LS(\mathcal{M}z_\infty)$ is $t(\mathcal{M}z_\infty)$-continuous. Lemma \ref{l_e5} implies that every derivation
$\overline{\delta}^{(z_j)}:=\overline{\delta}|_{LS(\mathcal{M}z_j)}:LS(\mathcal{M}z_j)\rightarrow LS(\mathcal{M}z_j)$ is also  $t(\mathcal{M}z_j)$-continuous for all $j\in J$. In this case, according to \cite[Cor.2.8]{BCS3}, the derivation $\overline{\delta}$ is  $t(\mathcal{M})$-continuous.
By Theorem \ref{t1} the derivation $\overline{\delta}$ is inner.
Repeating the proof of Theorem \ref{t4} we obtain that there exists an element $d\in\mathcal{E}$ such that $\delta(x)=[d,x]$ for all $x\in\mathcal{M}$.

Now, suppose that $\delta^*=\delta$. In this case,
$[d+d^*,x]=[d,x]-[d,x^*]^*=\delta(x)-(\delta(x^*))^*=\delta(x)-\delta^*(x)=0$
for any $x\in\mathcal{M}$. Consequently, the operator
$Re(d)=(d+d^*)/2$ commutes with every elements from $\mathcal{M}$, and by Proposition  \ref{p2_2}, $Re(d)$ is a central element in the algebra $LS(\mathcal{M})$. Therefore we may suggest that
$\delta(x)=[d,x],\ x\in\mathcal{M}$, where $d=ia,\ a\in
\mathcal{E}_h$. According to Theorem \ref{t_bs} there exist
$c=c^*$ from the centre of the algebra $LS(\mathcal{M})$ and a family
$\{u_\varepsilon\}_{\varepsilon>0}$ of unitary operators from
$\mathcal{M}$ such that
$$|[a,u_\varepsilon]|\geq (1-\varepsilon)|a-c|.$$
For $b=ia-ic$ and $\varepsilon=1/2$ we have $$|b|=|a-c|\leq
2|[a,u_{1/2}]|=2|[-id,u_{1/2}]|=2[d,u_{1/2}]\in\mathcal{E}.$$
Consequently, $b\in\mathcal{E}$ (see (\ref{e_solid})), moreover,
$$\delta(x)=[d,x]=[ia,x]=[b,x]$$
for all $x\in\mathcal{M}$. Since
$$(1-\varepsilon)|b|=(1-\varepsilon)|a-c|\stackrel{(\ref{e_bs})}{\leq} |[a,u_\varepsilon]|=|[d,u_\varepsilon]|=|\delta(u_\varepsilon)|,$$
it follows that
$$(1-\varepsilon)\|b\|_{\mathcal{E}}\stackrel{(\ref{e2})}{\leq} \|\delta(u_\varepsilon)\|_{\mathcal{E}}\leq \|\delta\|_{\mathcal{M}\rightarrow\mathcal{E}}$$
for all $\varepsilon>0$, that implies the inequality
$\|b\|_{\mathcal{E}}\leq
\|\delta\|_{\mathcal{M}\rightarrow\mathcal{E}}$.

If $\delta^*=-\delta$, then taking $Im(d)$ instead of $Re(d)$ and repeating previous proof we obtain that $\delta(x)=[b,x]$, where
$b\in\mathcal{E}$ and $\|b\|_{\mathcal{E}}\leq
\|\delta\|_{\mathcal{M}\rightarrow\mathcal{E}}$.

Now, suppose that $\delta\neq\delta^*$ and $\delta\neq
-\delta^*$. Equality (\ref{e1}) implies that
\begin{gather*}
\begin{split}
\|\delta^*\|_{\mathcal{M}\rightarrow\mathcal{E}}&=\sup\{\|\delta(x^*)^*\|_{\mathcal{E}}:\
\|x\|_{\mathcal{M}}\leq 1\}\\&= \sup\{\|\delta(x)\|_{\mathcal{E}}:\
\|x\|_{\mathcal{M}}\leq
1\}=\|\delta\|_{\mathcal{M}\rightarrow\mathcal{E}}.
\end{split}
\end{gather*}
Consequently,
$$\|Re(\delta)\|_{\mathcal{M}\rightarrow\mathcal{E}}=
2^{-1}\|\delta+\delta^*\|_{\mathcal{M}\rightarrow\mathcal{E}}\leq
\|\delta\|_{\mathcal{M}\rightarrow\mathcal{E}}.$$
Similarly,
$\|Im(\delta)\|_{\mathcal{M}\rightarrow\mathcal{E}}\leq
\|\delta\|_{\mathcal{M}\rightarrow\mathcal{E}}$. Since
$(Re(\delta))^*=Re(\delta)$, $(Im(\delta))^*=Im(\delta)$, there exist $d_1,d_2\in\mathcal{E}$, such that
$Re(\delta)(x)=[d_1,x],\ Im(\delta)(x)=[d_2,x]$ for all
$x\in\mathcal{M}$ and
$\|d_i\|_{\mathcal{E}}\leq\|\delta\|_{\mathcal{M}\rightarrow\mathcal{E}}$,
$i=1,2$. Taking $d=d_1+id_2$, we have that $d\in\mathcal{E}$,
$\delta(x)=(Re(\delta)+i\cdot
Im(\delta))(x)=[d_1,x]+i[d_2,x]=[d,x]$ for all $x\in\mathcal{M}$,
in addition $\|d\|_{\mathcal{E}}\leq
2\|\delta\|_{\mathcal{M}\rightarrow\mathcal{E}}$.
\end{proof}

Note that in \cite{B-S_d} in Theorem 16 it is given a variant of Theorem
 \ref{t_mbm} for a Banach $\mathcal{M}$-bimodule of locally
measurable operators with additional condition either one of separability or one of reflexivity. Theorem \ref{t_mbm}  given herein
removes all these assumptions from a Banach
$\mathcal{M}$-bimodule of locally measurable operators.

Let us point out one of important class of Banach $\mathcal{M}$-bimodules of locally
measurable operators connected with the theory of noncommutative integration.

Let $\mathcal{M}$ be a semifinite von Neumann algebra and $\tau$
be a faithful normal semifinite trace on $\mathcal{M}$. Let
$S(\mathcal{M},\tau)$ be the $*$-algebra of all $\tau$-measurable
operators affiliated with $\mathcal{M}$.

For each $x\in S(\mathcal{M},\tau)$ it is possible to define the
generalized singular value function
\begin{gather*}
\begin{split}
\mu_t(x)&=\inf\{\lambda>0:\tau(E_\lambda^\bot(|x|))\leq
t\}\\&=\inf\{\|x(\mathbf{1}-e)\|_\mathcal{M}:e\in\mathcal{P}(\mathcal{M}),\tau(e)\leq
t\},
\end{split}
\end{gather*}
which allows to define and  study a noncommutative version of
rearrangement invariant function spaces. It should be noted that
at the present time the theory of noncommutative rearrangement
invariant spaces has a significant place in researches of Banach
spaces (see e.g. \cite{DDdP}, \cite{K-S}).

Let $\mathcal{E}$ be a linear subspace in $S(\mathcal{M},\tau)$ equipped with a Banach norm $\|\cdot\|_{\mathcal{E}}$ with the following property:
$$\text{If}\ x\in S(\mathcal{M},\tau),\ y\in\mathcal{E}\
\text{and}\ \mu_t(x)\leq\mu_t(y)\ \text{then}\ x\in\mathcal{E}\
\text{and}\ \|x\|_{\mathcal{E}}\leq\|y\|_{\mathcal{E}}.$$ In this
case, the pair $(\mathcal{E},\|\cdot\|_{\mathcal{E}})$ is called
\emph{rearrangement invariant spaces of measurable operators}.
Every rearrangement invariant spaces  of measurable operators is a
Banach $\mathcal{M}$-bimodule \cite{DDdP}, and therefore Theorem
\ref{t_mbm} implies the following

\begin{corollary}
\label{c_last} Let $(\mathcal{E},\|\cdot\|_{\mathcal{E}})$ be a
rearrangement invariant spaces of measurable operators, affiliated
with a semifinite von Neumann algebra $\mathcal{M}$ and with a
faithful semifinite normal trace $\tau$. Then any derivation
$\delta:\mathcal{M}\rightarrow\mathcal{E}$ is continuous and
there exists $d\in\mathcal{E}$ such that $\delta(x)=[d,x]$ for all
$x\in\mathcal{M}$ and $\|d\|_{\mathcal{E}}\leq
2\|\delta\|_{\mathcal{M}\rightarrow\mathcal{E}}.$
\end{corollary}

\end{document}